\documentclass{article}

\usepackage[T1]{fontenc}
\usepackage[utf8]{inputenc}
\usepackage{lmodern}
\usepackage{microtype}

\usepackage[total={5.8in, 7.8in}, asymmetric, bindingoffset=0.4in]{geometry}
\usepackage{amsmath}
\usepackage{amssymb}
\usepackage{amsthm}
\usepackage{mathtools}
\usepackage{bbm}
\usepackage{enumerate}
\usepackage{enumitem}
\usepackage[dvipsnames]{xcolor}

\usepackage{multirow}
\usepackage{empheq}
\usepackage{listings}
\usepackage{color}

\usepackage[square,comma,numbers]{natbib}
\usepackage{hyperref}
\usepackage{graphicx} 
\usepackage[font=small,labelfont=bf]{caption}
\usepackage{tikz}
\usetikzlibrary{babel}
\usepackage{subfig}

\numberwithin{equation}{section}
\usepackage{pgfplots}
\pgfplotsset{width=10cm,compat=1.9}

\def\cF{{\mathcal F}}

\def\cH{{\mathcal H}}

\def\cK{{\mathcal K}}
\def\cL{{\mathcal L}}

\def\cP{{\mathcal P}}

\def\cV{{\mathcal V}}

\def\C{\mathbb{C}}

\def\E{\mathbb{E}}

\def\N{\mathbb{N}}

\def\P{\mathbb{P}}
\def\R{\mathbb{R}}

\def\T{\mathbb{T}}

\def\d{\mathrm{d}}

\def\Law{{\textstyle Law}}

\providecommand{\abs}[1]{\lvert#1\rvert}
\providecommand{\norm}[1]{\lVert#1\rVert}

\usepackage{MnSymbol}
\theoremstyle{plain}
\newtheorem{theorem}{Theorem}[section]
\newtheorem{lemma}[theorem]{Lemma}
\newtheorem{corollary}[theorem]{Corollary}
\newtheorem{proposition}[theorem]{Proposition}
\newtheorem{assumption}[theorem]{Assumption}

\newtheorem{remark}[theorem]{Remark}

\usepackage[symbol]{footmisc}

\usepackage{fancyhdr}
\date{\vspace{-1em}\normalsize{\today}}

\title{On the stability of the invariant probability measures of McKean-Vlasov equations}

\author{Quentin Cormier\footnote{
	Inria, CMAP, CNRS, \'Ecole polytechnique, Institut Polytechnique de Paris, 91120 Palaiseau, France.}}

 \date{\today}
\begin{document}
\maketitle
\vspace{5pt}

\abstract{
We study the long-time behavior of some McKean-Vlasov stochastic differential
equations used to model the evolution of large populations of interacting agents. We give conditions ensuring the local stability of an invariant probability measure. Lions derivatives are used in a novel way to obtain our stability criteria. We obtain
 results for non-local McKean-Vlasov equations on $\R^d$ and for McKean-Vlasov equations on the torus where the interaction kernel is given by a convolution. On $\R^d$, we prove that the location of the roots of an analytic function determines the stability. On the torus, our stability criterion involves the Fourier coefficients
of the interaction kernel. In both cases, we prove the convergence in the Wasserstein metric $W_1$ with an exponential rate of convergence.

}
\vspace{10pt}
\vspace{10pt}

\section{Introduction}
We are interested in the long-time behavior of the solutions of a class of McKean-Vlasov stochastic differential equations (SDE) of the form:
\begin{align} 
	\label{eq:McKeanVlasov}
	\d X^\nu_t &= \mathcal{V}(X^\nu_t, \mu_t) \d t + \sigma \d B_t,  \\
	\mu_t &= \Law(X^\nu_t),  \quad \mu_0 = \nu. \nonumber
\end{align}
In this equation, ${(B_t)}_{t \geq 0}$ is a standard $\mathbb{R}^d$-valued Brownian motion, $\sigma$ is a deterministic matrix, and $\nu$ is the law 
of the initial condition $X^\nu_0$, assumed to be independent of ${(B_t)}_{t \geq 0}$. 
McKean-Vlasov equations appear naturally  as the limit $N \rightarrow 
\infty$ of the following particle system ${(X^{i, N}_t)}_{t \geq 0}$, solution of
\begin{equation}
	\label{eq:particle system} 
\d X^{i, N}_t = \mathcal{V}(X^{i, N}_t, \mu^N_t) \d t + \sigma \d B^{i, N}_t, \qquad 1 \leq i \leq N,
\end{equation}
where $\mu^N_t$ is the empirical measure  $\mu^N_t = \frac{1}{N} \sum_{j  =1}^N { \delta_{X^{j, N}_t}}$ and 
${(B^{i, N}_t)}_{t \geq 0}$ are $N$ independent standard Brownian motions. 
We refer to~\cite{MR1108185} for an introduction to this topic.

Such particle systems and their mean-field counterparts are used in a wide range of applications such as plasma 
physics~\cite{MR541637, MR2278413}, fluid mechanics~\cite{MR3254330}, 
astrophysics (particles are stars or galaxies~\cite{MR1668556}), bio-sciences (to understand 
the collective behavior of animals~\cite{MR2860672}), neuroscience (to model assemblies of neurons, such as 
integrate and fire neurons~\cite{MR3349003, MR3573298} or FitzHugh–Nagumo neurons~\cite{MR4254489}), 
opinion 
dynamics~\cite{MR3305654} and economics~\cite{MR3752669}.

In these applications, one important question concerns the long-time behavior of the solutions. As such, 
the ergodic properties of McKean-Vlasov equations~\eqref{eq:McKeanVlasov} have been studied in many 
different 
contexts and approaches.

Two families of assumptions are known to ensure that~\eqref{eq:McKeanVlasov} admits a unique, globally attractive invariant probability measure. 
The first type of assumption deals with kernels given by $\mathcal{V}(x, \mu) = - \nabla V(x)  - \nabla W * 
\mu(x)$, 
where $V, W$ have suitable convexity properties. The first results in this direction were 
obtained in~\cite{MR1632193, MR1632197} in 
dimension one. In larger dimensions,~\cite{MR1970276, MR2208726}  proved the convergence uniformly in time of 
a suitable particle system towards the mean-field equation.  As such, they obtained the ergodicity of the 
McKean-Vlasov 
equation from the ergodicity of the particle system. These uniformly in time propagation of chaos arguments 
have been used wisely; see for instance~\cite{MR4020054, MR4333408} for recent results in this direction.
These results have also been obtained by using functional inequalities~\cite{MR1637274, MR2053570}: the idea is 
to 
define 
a measure-valued 
functional (known as the entropy or free energy), which only decreases along the trajectories of the 
solution 
of~\eqref{eq:McKeanVlasov}.

The second kind of assumption involves weak enough interactions. When the dependence of $\mathcal{V}$ with 
respect to the measure is sufficiently weak, one expects global stability because this situation can be seen as a 
perturbation of the case without interactions. As such, it is possible to extend techniques from ergodic Markov 
processes to the case of weak interactions. This includes, for instance, coupling techniques~\cite{ganzb2008, MR3403022, RevModPhys.77.137, Eberlecontraction, MR3939573} or Picard iterations in suitable spaces~\cite{MR4080722}.

It is also well-known that, in general, such global stability results cannot hold because~\eqref{eq:McKeanVlasov} may have multiple invariant probability measures and periodic solutions~\cite{MR781411, 
MR2639745, zbMATH06330199}.
These examples motivate the current question of the paper, namely the study of the local stability of a given invariant 
probability 
measure of~\eqref{eq:McKeanVlasov}. That is, being given $\nu_\infty$ an invariant probability measure of~\eqref{eq:McKeanVlasov}, we address the following question:

\textit{Is there exist an open neighborhood of $\nu_\infty$ such that for all initial 
conditions $\nu$ within this neighborhood, the law of $X^\nu_t$ converges to $\nu_\infty$, as $t$ goes to infinity?
If so, for which metric does the convergence hold, and what is the rate of convergence?}

Such local stability results can be obtained via partial differential equation (PDE) techniques, using that the 
marginals of the non-linear process solve a non-linear PDE (the Fokker-Planck equation). The strategy is to linearize 
the non-linear PDE around $\nu_\infty$, to study the existence of a spectral gap for the linear equation in 
appropriate Banach 
spaces, and to use perturbation techniques to obtain the convergence for the non-linear PDE\@.
We refer to~\cite{MR610244, MR2310258} for an overview of these techniques. 
When the non-linear PDE admits a gradient flow structure, it is also possible to study the local stability of an 
invariant probability measure using functional inequalities; see~\cite{MR0743525, MR3098681, MR4062483}. In~\cite{MR0743525}, the author study the local stability of an invariant probability measure in weighted $L^2$ norm. The result is obtained assuming a spectral condition related to the positivity of the Hessian of the free energy functional, evaluated at the invariant distribution.

Our approach differs from these two methods on several points. We do not rely on the non-linear Fokker-Planck 
PDE nor need a gradient flow structure. Instead, we use directly the stochastic representation~\eqref{eq:McKeanVlasov}.
Our strategy is to differentiate the interaction kernel with respect to the initial probability measure, in the 
neighborhood of $\nu_\infty$.
 There are several notions of derivation with respect to probability measures (see~\cite{MR3752669}): we use here the Lions derivatives. 
We denote by $\mathcal{P}_2(\mathbb{R}^d)$ the set of probability measures on $\mathbb{R}^d$ having a second 
moment.
For all $x \in \mathbb{R}^d$ and $t \geq 0$, 
we consider the function
\[ \mathcal{P}_2(\mathbb{R}^d) \ni \nu \mapsto \mathcal{V}(x, \Law(X^\nu_t)) =: v^x_t(\nu) \in \mathbb{R}^d 
,  \]
where $X^\nu_t$ is the solution of~\eqref{eq:McKeanVlasov} starting with law $\nu$ at time $0$. We prove that under 
suitable assumptions, this function is Lions differentiable at $\nu_\infty$, 
meaning that for all $\nu \in \mathcal{P}_2(\mathbb{R}^d)$, we have
\[ \mathcal{V}(x, \Law(X^\nu_t)) = \mathcal{V}(x, \nu_\infty) + \E \partial_\nu v^x_t(\nu_\infty)(Z_0) \cdot (Z-Z_0) + 
o( 
{(\E |Z-Z_0|^2)}^{1/2}).  \]
In this equation, $Z, Z_0$ are any random variables defined on 
the same 
probability space, with laws equal to $\nu$ and $\nu_\infty$. We write $\E (Z-Z_0 | Z_0) = h(Z_0)$, 
where $h$ is a deterministic function from 
$\mathbb{R}^d$ to $\mathbb{R}^d$.
As such, the function $h$ encodes the correlations between the initial conditions $Z$ and $Z_0$.
It follows from the Cauchy–Schwarz 
inequality that $\E |h(Z_0)|^2 \leq \E |Z-Z_0|^2 < \infty$.  Therefore, $h \in L^2(\nu_\infty)$. We define the 
linear operator $ \Omega_t: L^2(\nu_\infty) \rightarrow L^2(\nu_\infty)$ by
\[  \Omega_t(h) :=  x \mapsto \E \partial_\nu  v^x_t(\nu_\infty)(Z_0) \cdot h(Z_0). \]
The fact that $\Omega_t(h) \in L^2(\nu_\infty)$ for all $h \in L^2(\nu_\infty)$ is not granted apriori and will follow from our assumptions on the 
function $\mathcal{V}$. So we have (recall that $\nu = \Law(Z)$ and $\E (Z-Z_0 
| Z_0) = h(Z_0)$)
\[ \mathcal{V}(x, \Law(X^\nu_t)) = \mathcal{V}(x, \nu_\infty) + \Omega_t(h)(x)  + 
o( 
{(\E |Z-Z_0|^2)}^{1/2}).  \]
Our spectral conditions under which we prove that $\nu_\infty$ is locally stable can be stated in terms of the decay 
of 
the function $t \mapsto \Omega_t$, as $t$ goes to infinity. We show that the integrability 
of this function on $\mathbb{R}_+$ implies the stability of $\nu_\infty$. In addition, the decay of $t 
\mapsto \Omega_t$ as $t$ goes to infinity 
gives precisely the rate of convergence of $\Law(X^\nu_t)$ towards $\nu_\infty$, in Wasserstein metrics.
Crucial to our analysis, we provide an explicit integral equation to compute this function $\Omega_t$. To do so, we 
consider the linear process ${(Y^\nu_t)}_{t \geq 0}$ associated with~\eqref{eq:McKeanVlasov} and $\nu_\infty$, defined as the 
solution of 
\[ \d Y^\nu_t =  \mathcal{V}(Y^\nu_t, \nu_\infty)\d t + \sigma \d B_t, \]
starting from $\Law(Y^\nu_0) = \nu$. We define similarly for $x \in \mathbb{R}^d$ and $t \geq 0$ the 
function 
\[ \mathcal{P}_2(\mathbb{R}^d) \ni \nu \mapsto u^x_t(\nu) := \mathcal{V}(x, \Law(Y^\nu_t)). \]
Under non-restrictive assumptions, $u^x_t$ is Lions differentiable at $\nu_\infty$, and we can define
\[ \forall h \in L^2(\nu_\infty), \quad \Theta_t(h) := x \mapsto \E \partial_\nu u^x_t(\nu_\infty)(Z_0) \cdot h(Z_0). \]
We prove the following key relation between $\Theta_t$ and $\Omega_t$
\[ \forall t \geq 0, \quad  \Omega_t(h) = \Theta_t(h) + \int_0^t{ \Theta_{t-s}(\Omega_s(h)) \d s}.   \]
That is, $\Omega$ is a solution of a Volterra integral equation whose kernel is given by $\Theta$: in the language 
of integral equations, $\Omega$ is the resolvent of $\Theta$. 
This relation is helpful because it is easier to get estimates on $u^x_t$, which involves a linear 
Markov  
process, rather than getting estimates on $v^x_t$, which involves the solution of the McKean-Vlasov equation~\eqref{eq:McKeanVlasov}. 
In particular, this relation allows to deduce the decay properties of $\Omega$ from properties of $\Theta$, using Laplace transform.
We obtain our stability results for the Wasserstein $W_1$ metric.

The contributions of this work are the following.
First, in Section~\ref{sec:part1}, we consider dynamics of the form $\cV(x, \mu) = b(x) + \int_{\R^d} F(x, y) \mu (\d y)$, for some smooth functions $b: \R^d \rightarrow \R^d$ and $F: \R^d \times \R^d \rightarrow \R^d$. The function $b$ is assumed to be confining.
Our main result, Theorem~\ref{th:main result}, states that the stability of an invariant probability measure is determined by the location of the roots of an explicit analytic function associated with the dynamics. Stability holds when all the roots lie on the left half-plane, and we prove convergence in Wassertein metric $W_1$ with an exponential rate.
Our result shows that the stability is completely determined by a discrete set, this set being given by the zeros of an analytic function associated to the underlying Markov process.
Note that we do not require any structural assumption on $b$ and $F$: in particular, we do not require any convexity assumption on the coefficients. 
Our stability criterion is analogous to the Jacobian stability criterion for ODE, for which the location of the zeros of the characteristic polynomial determines the stability. 

Second, in Section~\ref{sec:part2}, we consider a McKean-Vlasov equation on the torus $\T^d := {(\R / (2 \pi \mathbb{Z}))}^d$, with an interaction kernel given by a convolution: $\cV(x, \mu)=-\int_{\T^d}\nabla W(x-y)\mu(\d y)$, where $W$ is a smooth function from $\T^d$ to $\R$. We assume that $\sigma = \sqrt{2 \beta^{-1}} I_d$ for some $\beta > 0$, where $I_d$ is the identity matrix. This setting covers many interesting models; see~\cite{MR4062483}. We study the stability of the uniform probability measure $U(\d x) := \frac{\d x}{{(2 \pi)}^d}$. Our second main result, Theorem~\ref{th:second main result}, states that when $\inf_{n \in \mathbb{Z}^d \setminus \{0\}} |n|^2(\beta + \tilde{W}(n)) > 0$, $\tilde{W}(n)$ being the $n$-th Fourier coefficient of $W$, then $U$ is locally stable for the $W_1$ metric.
Our result complements the results of~\cite{MR4062483}, for which static bifurcations are studied: in particular, we exhibit the same critical parameter.
In both parts, we use the strategy described above, using Lions derivatives and probabilistic tools.
The criteria we obtain are optimal: violations of the criteria occur strictly at bifurcation points.

The strategy presented in this work also applies to mean-field models of noisy 
integrate-and-fire neurons. In an unpublished preliminary version of this work~\cite{cormier2021meanfield}, we 
study the stability of the stationary solutions of such a mean-field model of noisy neurons. In addition, periodic solutions via Hopf bifurcations are studied in~\cite{MR4316639}. For the sake of clarity, 
we restrict here ourselves to a diffusive setting.

Finally, we mention an important open problem concerning the long-time behavior of the particle
system~\eqref{eq:particle system}.  On the one hand, general conditions are known to ensure that the particle system is 
ergodic. On the other hand, numerical studies show that this particle system can have metastable 
behavior in the sense that the convergence of the empirical 
measure $\mu^N_t$ towards its invariant state can be very slow  when $N$ is large.
The locally stable invariant probability measures of the non-linear equation~\eqref{eq:McKeanVlasov} are 
good 
candidates to be metastable states of the particle system~\eqref{eq:particle system}. Characterizing those 
metastable states in quantitative terms is a challenging mathematical question. Recent partial results have been 
obtained in this direction~\cite{zbMATH06380861, MR4187123, lcherbach2020metastability, zbMATH07493825, delarue2021uniform}, and we hope to progress on this question in future works.

\textbf{Acknowledgments}. The author would like to express his gratitude to Etienne Tanré for many 
suggestions at different steps of this work and hearty encouragement. He also thanks Romain Veltz and René 
Carmona for their valuable advice.
This research has received funding from the 
European Union’s Horizon 2020 Framework Programme for Research and Innovation under the Specific Grant 
Agreement No. 945539 (Human Brain Project
SGA3) and was supported by AFOSR FA9550-19-1-0291.

\section{McKean-Vlasov equations on \texorpdfstring{$\mathbb{R}^d$}{Rd}}\label{sec:part1}
\subsection{Main result}
Let $\cP_1(\R^d)$ be the space of probability measures on $\R^d$ with a finite first moment.
We consider the following McKean-Vlasov equation on $\R^d$:
\begin{equation}
\label{eq:McKeanVlasov1}
\d X^\nu_t = b(X^\nu_t) \d t + \int_{\R^d}{F(X^\nu_t, y) \mu_t(\d y)}\d t + \sigma \d B_t \quad \text{ with } \quad \mu_t = \Law(X^\nu_t). 
\end{equation}
The initial condition $X^\nu_0$ has law $\nu \in \mathcal{P}_1(\R^d)$. 
Here, ${(B_t)}_{t \geq 0}$ is a $d$-dimensional standard Brownian motion, $\sigma \in M_d(\R)$ is a constant $d \times d$ matrix with $\det \sigma > 0$. 
\begin{assumption}\label{ass:standing assumptions1}
	The functions $b: \R^d \rightarrow \R^d$ and  $F: \R^{2d} \rightarrow \R^d$ are $C^2$, $b$ is globally Lipschitz, and the derivatives of $F$ are bounded ($b$ and $F$ are not assumed to be bounded themselves):
	\[ \forall i,j \in {\{i,\dots,d\}}^2, \quad \norm{ \partial_{x_i} F}_\infty + \norm{ \partial^2_{x_i,x_j} F}_\infty < \infty.\] 
\end{assumption}
This ensures in particular that~\eqref{eq:McKeanVlasov1} has a unique path-wise solution.
Let $\nu_\infty \in \cP_1(\R^d)$ be an invariant probability measure of~\eqref{eq:McKeanVlasov1}, that is:
\[ \forall t \geq 0, \quad \Law(X^{\nu_\infty}_t) = \nu_\infty. \]
Denote by
$\alpha(x)$ the interaction term under $\nu_\infty$:
\begin{equation}\label{eq:definition of alpha}
	\forall x \in \R^d, \quad \alpha(x) := \int_{\R^d} F(x, y) \nu_\infty(\d y).
\end{equation}
Each invariant probability measure of~\eqref{eq:McKeanVlasov1} is characterized by its associated
function $\alpha$, and we sometimes denote by $\nu^\alpha_\infty$ such invariant probability measure to emphasize
the dependence on $\alpha$. We assume:
\begin{assumption}\label{ass:standing assumption2}
	There exists $\beta > 0$ and $R \geq 0$ such that
	\[  \forall x,x' \in \R^d, \quad |x-x'| \geq R \implies (x-x') \cdot \left[(b + \alpha)(x) - (b + \alpha)(x') \right]\leq -\beta |x-x'|^2. \]
\end{assumption}
\begin{remark}
	In particular, this is satisfied provided that $F$ is bounded (or independent of $x$) and there exists $\beta > 0$ and $R \geq 0$ such that:
	\[ |x-x'| \geq R \implies (x-x') \cdot (b(x) - b(x')) \leq -\beta |x-x'|^2. \]
	\end{remark}
	Let ${(Y^\alpha_t)}_{t \geq 0}$ the solution of the linear SDE
\begin{equation}
        \label{eq:def of Y alpha}
	\d Y^\alpha_t =  b(Y^{\alpha}_t) \d t + \alpha(Y^\alpha_t) \d t + \sigma \d B_t.
\end{equation}
Note that $\nu_\infty = \nu^\alpha_\infty$ is also an invariant probability measure of this linear SDE\@. In addition, a result of Eberle~\cite{Eberlecontraction} (see~\eqref{eq:W1 bound eberle} below) ensures that $\nu_\infty$ is the unique invariant probability measure of~\eqref{eq:def of Y alpha}. 

Consider $\cH := L^2(\nu_\infty)$ the Hilbert space of measurable functions $h: \R^d \rightarrow \R^d$ satisfying:
\[  \norm{h}^2_{\cH} := \int |h(y)|^2 \nu_\infty(\d y) < \infty. \]
We denote by $\cL(\cH)$ the space of bounded linear operators from $\cH$ to itself.
Key to our analysis is the following family of bounded linear operators $\Theta_t \in \cL(\cH)$, $t \geq 0$:
\begin{align} 
	\forall h \in \cH,\quad  &\Theta_t(h)(x) := \int_{\mathbb{R}^d}{\nabla_y \E_y F(x, Y^{\alpha}_t) \cdot h(y) \nu^\alpha_\infty(\d y). } 
	\label{eq: definition of Theta t i j}
\end{align}
The notation $\E_y F(x, Y^{\alpha}_t)$  means that the initial condition of $(Y^\alpha_t)
$ is set to be $y \in
\mathbb{R}^d$ (that is $Y^\alpha_0 = y$). In addition, $\nabla_y \E_y F(x, Y^{\alpha}_t)$ is the Jacobian matrix of $y \mapsto \E_y F(x, Y^{\alpha}_t)$.
The result of Eberle~\cite{Eberlecontraction} (see~\eqref{eq:gradient bound} below) implies that there exists $\kappa_* > 0$ and $C > 0$ such that
\begin{equation}\label{eq:control infty norm theta}
	\forall h \in \cH, \quad \norm{\Theta_t(h)}_\infty \leq C e^{-\kappa_* t} \norm{h}_{\cH}. 
\end{equation}
Denote by $I \in \cL(\cH)$ the identity operator.
Let $D = \{z \in \C,~\Re(z) > -\kappa_* \}$. For all $z \in D$, consider the Laplace transform of $\Theta_t$:
\[ \hat{\Theta}(z)= \int_0^\infty{e^{-zt} \Theta_t \d t }. \]
We recall some properties of this operator-valued function.
\begin{proposition}
We have:
\begin{enumerate}[label=~(\alph*)]
	\item For all $z \in D$, $\hat{\Theta}(z) \in \cL(\cH)$ is an Hilbert-Schmidt operator. In particular, $\hat{\Theta}(z)$ is a compact operator. 
	\item $D \ni z \mapsto \hat{\Theta}(z)$ is an analytic operator-valued function, and ${(I - \hat{\Theta}(z))}^{-1}$ exists for all $z \in D \setminus S$, where $S$ is a discrete subset of $D$. In addition, ${(I-\hat{\Theta}(z))}^{-1}$ is meromorphic in $D$, holomorphic in $D \setminus S$. 
	\item The function $D \ni z  \mapsto \det(I - \hat{\Theta}(z)) \in \C$ is analytic, where $\det$ is the regularized Fredholm determinant. In addition, $I-\hat{\Theta}(z)$ is invertible if and only if $\det(I - \hat{\Theta}(z)) \neq 0$.     
\end{enumerate}
\end{proposition}
\begin{proof}
	By Fubini, we have	\[ \hat{\Theta}(z)(h) = x \mapsto \int_{\R^d} K_z(x, y) h(y) \nu_\infty(\d y), \] 
	where $K_z(x, y) := \int_0^\infty{e^{-zt} \nabla_y \E_y F(x, Y^\alpha_t) \d t}$.
	Using~\eqref{eq:gradient bound}, it holds that $\sup_{x, y \in \R^d}\abs{K_z(x, y)} < \infty$ for all $z \in D$. Therefore, $K_z(\cdot, \cdot) \in L^2(\R^d \times \R^d, \nu_\infty \otimes \nu_\infty)$ and so Theorem VI.23 in~\cite{SimonBarryBook} applies and gives the first point.
	For $\Re(z)$ large enough, it holds that $\norm{\hat{\Theta}(z)}_{\cH} < 1$. 
	Therefore, $I - \hat{\Theta}(z)$ is invertible provided that $\Re(z)$ is large enough, the inverse is given by its Neumann series. So the second point follows from the analytic Fredholm theorem, see Theorem VI.14 in~\cite{SimonBarryBook}. Finally, the third point is a fundamental result of the theory of Fredholm determinants for Hilbert-Schmidt operators, see for instance~\cite{simonInfiniteDeterminant}.  
\end{proof}
Our main result is
\begin{theorem}\label{th:main result}
	Consider $\nu_\infty$ an invariant probability measure of~\eqref{eq:McKeanVlasov1} and let $\alpha$ be given by~\eqref{eq:definition of alpha}. Assume that Assumptions~\ref{ass:standing assumptions1} and~\ref{ass:standing assumption2} hold. Define the ``abscissa'' of the rightmost zeros of $\det(I - \hat{\Theta}(z))$:
	\begin{equation}
		\label{eq:spectral assumption}
	-\lambda' := \sup \{\Re(z)~|~ z \in D, ~ \det(I - \hat{\Theta}(z))  = 0 \}. \end{equation}
	Assume that $\lambda'> 0$. Then $\nu_\infty$ is locally stable: there exists $C, \epsilon > 0$ and $\lambda \in (0, \lambda')$ such that for all $\nu \in \cP_1(\R^d)$ with $W_1(\nu, \nu_\infty) < \epsilon$, it holds that
	\[  \forall t \geq 0, \quad  W_1(\Law(X^\nu_t), \nu_\infty)  \leq C W_1(\nu, \nu_\infty)
        e^{-\lambda t}.\]
\end{theorem}
\subsection{Remarks and examples}
We now give explanations on Theorem~\ref{th:main result}, in particular on the spectral assumption involving~\eqref{eq:spectral assumption}.
From now on, the constants may vary from one line to the other.
\subsubsection*{Gradients bounds}
We denote by $(Y^{\alpha, \delta_x}_t)$ the solution of~\eqref{eq:def of Y alpha} with initial condition $Y^{\alpha, \delta_x}_0 = x$.  
Under assumptions~\ref{ass:standing assumption2}, Theorem 1 in~\cite{Eberlecontraction} applies: there exists $\kappa_* > 0$ and $C_* > 1$ such that for all $x, y \in \R^d$ and all $t \geq 0$,
\begin{equation}\label{eq:W1 bound eberle}
	W_1(\Law(Y^{\alpha, \delta_x}_t), \Law(Y^{\alpha, \delta_y}_t)) \leq C_* e^{-\kappa_* t} |x-y|.
\end{equation}
We deduce from this inequality the following gradient bound. For all $f \in C^1(\R^d)$:  
\begin{equation}\label{eq:gradient bound}
	\forall y \in \R^d, \quad | \nabla_y \E_y f(Y^\alpha_t) | \leq C_* \lVert \nabla f \rVert_\infty e^{-\kappa_* t}, 
\end{equation}
In particular, by choosing $f = F(x, \cdot)$, we obtain the estimate~\eqref{eq:control infty norm theta}.
Note that it is possible to get gradient bounds similar to~\eqref{eq:gradient bound} under less restrictive assumptions on $b$ and $\sigma$; see~\cite{Talay90}.
\subsubsection*{On the spectral condition}
We consider the following bounded linear operator $\Omega_t \in \cL(\cH)$ by taking the Neumann series:
\begin{equation} \forall h \in  \cH, \quad \Omega_t(h) := \sum_{i \geq 1}{ \Theta^{\otimes i}_t (h)}, 
	\label{eq:link between Omega and Theta}
\end{equation}
where the linear operators $\Theta^{\otimes (i)}_t$ are defined recursively by
\[ \forall t \geq 0, \quad  \Theta^{\otimes (i+1)}_t(h)  = \int_0^t{ \Theta_{t-s}(\Theta^{\otimes i}_s(h)) \d s}, \quad 
\text{ and } \quad \Theta^{\otimes 1}_t(h) = \Theta_t(h).  \]
The series~\eqref{eq:link between Omega and Theta} converges uniformly on any compact $[0, T]$ for $T > 0$. The operators $\Omega_t$ and $\Theta_t$ satisfy the following Volterra integral equation:
\begin{align}
	\forall h \in \cH,  \quad \Omega_t(h) &= \Theta_t(h) + \int_0^t{ \Theta_{t-s}(\Omega_s(h))\d s}  	
	\label{eq:volterra equation Theta Omega}	\\
	&=  \Theta_t(h) + \int_0^t{ \Omega_{t-s}(\Theta_s(h)) \d s}. \nonumber
\end{align}
We denote by $\norm{\Omega_t}_{\cL(\cH)}$ the operator norm of $\Omega_t$:
\[  \norm{\Omega_t}_{\cL(\cH)} := \sup_{\norm{h}_{\cH} \leq 1} \norm{ \Omega_t(h) }_{\cH}.\]
Then the spectral condition $\lambda' > 0$ is equivalent to the exponential decay of $t \mapsto  \norm{\Omega_t}_{\cL(\cH)}$:
\begin{proposition}\label{prop:paley-wiener}
The two following statements are equivalent
\begin{enumerate}[label=~(\alph*)]
	\item $\lambda' > 0$, where $\lambda'$ is given by~\eqref{eq:spectral assumption}.
	\item $\exists \lambda > 0$ such that $\sup_{t \geq 0} e^{\lambda t} \norm{\Omega_t}_{\cL(\cH)} < \infty$.
\end{enumerate}
\end{proposition}
\begin{proof}
	We first show that (b) implies (a): by assumption, there exists $\lambda \in (0, \kappa_*)$ such that $z \mapsto \hat{\Omega}(z)$ is an analytic operator-valued function on $\Re(z) > -\lambda$. In view of~\eqref{eq:volterra equation Theta Omega}, we have for $\Re(z) > -\lambda$, $\hat{\Omega}(z) =  \hat{\Theta}(z) + \hat{\Theta}(z) \cdot \hat{\Omega}(z) =  \hat{\Theta}(z) + \hat{\Omega}(z) \cdot \hat{\Theta}(z)$, and so 
	\[  I = (I + \hat{\Omega}(z))(I - \hat{\Theta}(z)) = (I - \hat{\Theta}(z))(I + \hat{\Omega}(z)). \]
	Therefore, $I-\hat{\Theta}(z)$ is invertible for all $\Re(z) > -\lambda$, and consequently $\lambda' \geq \lambda > 0$.

	We then show that (a) implies (b): this follows from a Paley-Wiener theorem. Let $\lambda \in (0, \lambda')$ and define $K_t:= e^{\lambda t} \Theta_t$ and $R_t:= e^{\lambda t} \Omega_t$. It holds that $K \in L^1(\R_+; \cL(\cH))$ (because $\lambda < \kappa_*$) and $I - \hat{K}(z)$ is invertible for all $\Re(z) \geq 0$ (because $\lambda < \lambda'$).  
	Therefore, by~\cite[Ch. 2, Th. 4.1]{MR1050319}\footnote{The result in~\cite{MR1050319} is stated and proved for matrix valued operators. The extension to Hilbert-Schmidt operators is straightforward by the exact same arguments. }, it holds that $R \in L^1(\R_+; \cL(\cH))$.
	Using the estimate~\eqref{eq:control infty norm theta} and~\eqref{eq:volterra equation Theta Omega}, we find that
	\begin{align*}
		\norm{\Omega_t(h)}_\infty &  \leq C e^{- \kappa_* t} \norm{h}_{\cH} + C \int_0^t{ 
	e^{-\kappa_* s } 
	\norm{\Omega_{t-s}}_{\cL(\cH)} \norm{h}_{\cH}  \d s 
	} \\
	& \leq C \norm{h}_{\cH} \left( e^{-\kappa_*t} + e^{-\lambda t} \int_0^t{ e^{\lambda(t-s)} \norm{\Omega_{t-s}}_{\cL(\cH)} \d s } \right) \\
	& \leq C \norm{h}_{\cH}\left( 1 + \int_0^\infty{  \norm{R_s}_{\cL(\cH)} \d s } \right) e^{-\lambda t}.
	\end{align*}
	This shows that $\sup_{t \geq 0} e^{\lambda t} \norm{\Omega_t}_{\cL(\cH)} < \infty$.  This ends the proof.
\end{proof}
\subsubsection*{The spectral condition is necessary}
Let $h \in \cH$ be fixed and let $Z_0$ be a random variable of law $\nu_\infty$.
For $\epsilon \in \R$, let $\nu_\epsilon := \Law(Z_0 + \epsilon h(Z_0))$.
We will see that the operator $\Omega_t$ admits the following probabilistic representation (see Remark~\ref{rk:interpretation Theta_t, Omega_t} below):
\begin{equation}
	\label{eq:representation Omega}
\forall t \geq 0, \quad  \Omega_t(h)(x) = \lim_{\epsilon \rightarrow 0} \frac{1}{\epsilon} \int_{\R^d} { F(x, y) (\Law(X^{\nu_\epsilon}_t) - \nu_\infty)(\d y)}, \quad x \in \R^d,  
\end{equation}
where $(X^{\nu_\epsilon}_t)$ is the solution of the McKean-Vlasov equation~\eqref{eq:McKeanVlasov1} starting with law $\nu_\epsilon$.
We have by the Cauchy-Schwarz inequality 
\[ W_1(\nu_\epsilon, \nu_\infty) \leq \epsilon \E \abs{h(Z_0)} \leq \epsilon \norm{h}_{\cH}.  \]
Assume that the conclusion of Theorem~\ref{th:main result} holds. Then there exists $C$, $\lambda > 0$ such that for all $\epsilon$ small enough,
\[ W_1(\Law(X^{\nu_\epsilon}_t), \nu_\infty) \leq C \epsilon e^{-\lambda t}  \norm{h}_{\cH}.  \]
We deduce that:
\[ \forall x \in \R^d, \quad \left|  \int_{\R^d} { F(x, y) (\Law(X^{\nu_\epsilon}_t) - \nu_\infty)(\d y)} \right| \leq C \epsilon \norm{\nabla_y F}_\infty e^{-\lambda t} \norm{h}_{\cH}.  \]
Therefore, 
\[ \norm{ \Omega_t(h) }_\infty \leq C  \norm{\nabla_y F}_\infty e^{-\lambda t} \norm{h}_{\cH}. \]
Using Proposition~\ref{prop:paley-wiener}, we deduce that $\lambda' > 0$, where $\lambda'$ is given by~\eqref{eq:spectral assumption}. In other words, the spectral condition $\lambda' > 0$ is necessary to have stability in $W_1$ norm.
\subsubsection*{Structural stability with convex coefficients}
The classical situation where stability is known to hold globally is the following: assume that there exist functions $V, F: \R^d \rightarrow \R$ such that $b(x) = -\nabla V(x) \quad \text{ and } \quad F(x, y) = -\nabla W(x-y)$. 
Assume that $W$ is convex and even and that $V$ is strongly convex:
\begin{equation}
	\label{eq:W,V convex}
	\exists \theta > 0, \forall x \in \R^d, \quad W(x) = W(-x), \quad \nabla^2 W(x) \geq 0, \quad \text{ and } \quad  \nabla^2 V(x) \geq \theta I_d. 
\end{equation}
Then, by~\cite[Th. 7]{zbMATH06892760}, it holds that
\[ \forall t \geq 0, \forall \nu, \mu \in \cP_2(\R^d),\quad W_2(\Law(X^\nu_t), \Law(X^\mu_t)) \leq e^{-\theta t / 2} W_2(\nu, \mu). \] 
Therefore, these structural conditions ensure that~\eqref{eq:McKeanVlasov1} has a unique invariant probability measure, which is globally stable in $W_2$ norm. 
We show that our spectral condition is satisfied under these conditions.
\begin{lemma}
	Under the structural assumption~\eqref{eq:W,V convex}, it holds that $\lambda' > 0$, $\lambda'$ given by~\eqref{eq:spectral assumption}, and so Theorem~\ref{th:main result} applies.
\end{lemma}
\begin{proof}
	Let $\nu_\infty$ be the unique invariant probability measure of~\eqref{eq:McKeanVlasov1}. Let $h \in \cH, \epsilon \in \R$ and let $\nu_\epsilon = \cL(Z_0 + \epsilon h(Z_0))$, where $\Law(Z_0) = \nu_\infty$.
We have
\begin{align*}
	\left|\int_{\R^d} { \nabla W(x-y) (\Law(X^{\nu_\epsilon}_t) - \nu_\infty)(\d y)} \right| & \leq \norm{\nabla^2 W}_\infty W_1(\Law(X^{\nu_\epsilon}_t), \nu_\infty) \\
												   & \leq \norm{\nabla^2 W}_\infty W_2(\Law(X^{\nu_\epsilon}_t), \nu_\infty) \\
												   & \leq \norm{\nabla^2 W}_\infty e^{-\theta t /2 }W_2(\nu_\epsilon, \nu_\infty) \\
												   & \leq  C e^{-\theta t /2} \sqrt{\E \abs{\epsilon h(Z_0)}^2} = C |\epsilon| e^{-\theta t /2} \norm{h}_{\cH}.
\end{align*}
Therefore, by~\eqref{eq:representation Omega}, we have
\[ \forall h  \in \cH, \quad \norm{\Omega_t(h)}_\infty \leq C e^{-\theta t/2} \norm{h}_{\cH}. \]
So Proposition~\ref{prop:paley-wiener} applies and $\lambda' > 0$.
\end{proof}
\subsubsection*{Case of weak interactions}
One way to check that the spectral condition $\lambda' > 0$ holds, $\lambda'$ given by~\eqref{eq:spectral assumption}, is to compute the $L^1$ norm of $\Theta_t$. Recall that
$\norm{\Theta_t}_{\cL(\cH)} = \sup_{\norm{h}_{\cH} \leq 1} \norm{\Theta_t(h)}_{\cH}$.
\begin{lemma}
	Assume that $\int_0^\infty{ \norm{\Theta_t}_{\cL(\cH)} \d t} < 1$. Then $\lambda' > 0$ and so $\nu_\infty$ is
        locally stable.
\end{lemma}
\begin{proof}
	Recall that by~\eqref{eq:control infty norm theta}, it holds that $\sup_{t \geq 0} e^{\kappa_* t} \norm{\Theta_t}_{\cL(\cH)} < \infty$. Therefore,	there exists $\delta > 0$ small enough such that
\[ \int_0^\infty{ e^{\delta t} \norm{\Theta_t}_{\cL(\cH)} \d t} < 1. \]
For $\Re(z) \geq -\delta$, it holds that $\norm{\hat{\Theta}(z)}_{\cL(\cH)} \leq \int_0^\infty{e^{-\Re(z)t } \norm{\Theta_t}_{\cL(\cH)} \d t } < 1$.  
We deduce that $I - \hat{\Theta}(z)$ is invertible for $\Re(z) \geq -\delta$, with inverse given by $\sum_{k \geq 0} {(\hat{\Theta}(z))}^k$. So $\lambda' \geq \delta > 0$.
\end{proof}
This assumption is typically satisfied if the non-linear part in~\eqref{eq:McKeanVlasov1} is weak enough.  
Given $M \in M_d(\C)$, let $\norm{M}_2$ be the spectral norm of the matrix $M$. We let
$[F] := \sup_{x,y \in \R^d} \norm{\nabla_y F(x, y)}_2$. 
Then, from~\eqref{eq:W1 bound eberle}, we have $\norm{\Theta_t}_{\cL(\cH)} \leq C_* [F] e^{-\kappa_* t}$.
Therefore, if $[F] < \kappa_* / C_*$, then $\lambda' >0$, and so any invariant probability measure of~\eqref{eq:McKeanVlasov1} is locally stable.
\subsubsection*{Case of ``separable'' interactions}
Assume that $F$ is ``separable'', in the sense that there exist functions $w_i: \R^d \rightarrow \R^d$ and $f_i: \R^d \rightarrow \R$ such that
\[ \forall x,y \in \R^d \times \R^d, \quad F(x,y) = \sum_{i = 1}^p f_i(y) w_i(x). \]
Let $\cH_0$ be the following subspace of $\cH$ of finite dimension
\[ \cH_0 := \{ h~|~h = \sum_{i = 1}^p \beta_i w_i, ~\beta \in \R^p\}. \]
For all $h \in \cH$ and for all $t \geq 0$, it holds that $\Theta_t(h) \in \cH_0$. The restriction of $\Theta_t$ to $\cH_0$ can be represented by a $p \times p$ matrix, again denoted by $\Theta_t$, and we have
\[ \Theta^{i,j}_t = \int_{\R^d} \nabla_y \E_y f_i(Y^\alpha_t) \cdot w_j(y) \nu_\infty(\d y), \qquad 1\leq i,j\leq p.   \]
In that case, the determinant in Theorem~\ref{th:main result} is the standard determinant of matrices. This instance already covers a number of interesting examples. 
\subsubsection*{Case with small noise ($\sigma \simeq 0$)}
We now discuss the case where the noise $\sigma$ is small. 
	The case $\sigma \equiv 0$ would require a special treatment and is not included in Theorem~\ref{th:main result}. It is however instructive to see that the criterion involving~\eqref{eq:spectral assumption} is equivalent to the classical stability criterion of a deterministic dynamical systems in $\mathbb{R}^d$. When $\sigma \equiv 0$, the invariant measures are of the form $\delta_{x_*}$ for some $x_* \in \mathbb{R}^d$. Let $\cV(x,y) = b(x) + F(x,y)$. Assume that
	\begin{equation} 
		\label{eq:ass-Theta0}
		\exists \kappa > 0, \quad \sup_{t \geq 0} e^{\kappa t}  \norm{ e^{t \nabla_x \cV(x_*, x_*)}}_2 < \infty.  
	\end{equation}
	This condition means that $x_*$ is a stable equilibrium point for the ODE $\dot{y} = \cV(y, x_*)$. Indeed,~\eqref{eq:ass-Theta0} is equivalent to the fact that the Jacobian matrix of the vector field $x \mapsto \cV(x, x_*)$ has all its eigenvalues in the left half-plane $\Re(z) < 0$.
	\begin{lemma}\label{lem:sigma is equal to zero}
		Assume~\eqref{eq:ass-Theta0} holds and $\sigma \equiv 0$. Then, the criterion $\lambda' > 0$, $\lambda'$ given by~\eqref{eq:spectral assumption}, is equivalent to the fact that the Jacobian matrix of the vector field $x \mapsto \cV(x, x)$ at the point $x_*$ has all its eigenvalues in the left half-plane $\Re(z)< 0$.
	\end{lemma}
The proof is given in the Appendix. We now treat the case $\sigma$ small using a perturbation argument of $\sigma \equiv 0$.
Assume that for $\norm{\sigma}_2$ small enough, the McKean-Vlasov equation~\eqref{eq:McKeanVlasov} has a unique invariant measure in the neighborhood of $\delta_{x_*}$. That is, we assume that there exists $\rho_0, \sigma_0 > 0$ such that for any $\sigma \in M_d(\R)$ with $\det(\sigma) > 0$ and $\norm{\sigma}_2 \leq \sigma_0$, it holds that: 
\[ \text{Card} \{\nu \in \cP_1(\R^d):~ W_1(\nu, \delta_{x_*}) \leq \rho_0 \text{ and $\nu$ is an invariant probability measure of~\eqref{eq:McKeanVlasov}} \} = 1. \] 
We denote by $\nu^\sigma_\infty$ be this invariant distribution.
To emphasize the dependence on $\sigma$, until the end of this section, we also denote by $\Theta^\sigma_t$ the operator given by~\eqref{eq: definition of Theta t i j} when the SDE~\eqref{eq:def of Y alpha} has diffusion coefficient $\sigma$. Similarly, we denote by $\Omega^\sigma_t$ the operator defined by~\eqref{eq:link between Omega and Theta}.
The precise description of $\Theta^0_t$ and $\Omega^0_t$ is given in the proof of Lemma~\ref{lem:sigma is equal to zero}. In particular, we have
\[  \Theta^0_t(h)(x) = \nabla_y \cV(x, x_*) e^{t \nabla_x \cV(x_*, x_*)} \cdot h(x_*). \]
Assume that $x_*$ is a stable point for the ODE $\dot x = \cV(x, x)$. Then, by~\eqref{eq:ass-Theta0} and Lemma~\ref{lem:sigma is equal to zero},  it holds that
\begin{equation}
\label{eq:ass-Theta0-Omega0}
	\exists \kappa> 0, \quad \sup_{t \geq 0} e^{\kappa t} \left[  \norm{\Theta^0_t}_{\cL(\cH)} +  \norm{\Omega^0_t}_{\cL(\cH)} \right]< \infty.  
\end{equation}
\begin{proposition}\label{prop:sigma-small-stability}
	Assume that~\eqref{eq:ass-Theta0-Omega0} holds. In addition assume that for all $\epsilon > 0$, there exists $\theta > 0$ such that
	\begin{equation}
		\label{eq:stability-Omega-small-sigma-ass}
		\norm{\sigma}_2 \leq \theta \implies  \sup_{t \geq 0} e^{\kappa t} \norm{\Theta^\sigma_t - \Theta^0_t}_{\cL(\cH)} \leq \epsilon.  
	\end{equation}
Then there exists $\sigma_0, \lambda > 0$ such that
\[  \norm{\sigma}_2 \leq \sigma_0, \quad  \sup_{t \geq 0}  e^{\lambda t}  \norm{\Omega^\sigma_t}_{\cL(\cH)}< \infty. \]
Therefore Theorem~\ref{th:main result} applies and $\nu^\sigma_\infty$ is locally stable.
\end{proposition}
This result can be summarize has follows: assuming that $x_*$ is a stable equilibrium for the ODE $\dot{x} = \cV(x, x)$, if~\eqref{eq:stability-Omega-small-sigma-ass} holds then stability also hold with small noise.
Again, the proof is given on the Appendix.
Finally, let us comment on how the rate of convergence depends on $\sigma$. Under the assumptions of Proposition~\ref{prop:sigma-small-stability}, the constant $\lambda'$ given by~\eqref{eq:spectral assumption} is lower bounded by a strictly  positive constant independent of $\norm{\sigma}_2 \in [0, \sigma_0]$. However, the rate of convergence is $W_1$ norm still depends on $\sigma$: the bottleneck is coming from the constant $\kappa_*$ of~\eqref{eq:W1 bound eberle}, which in general vanishes as the noise goes to zero. In some specific examples, the rate of convergence does not vanish as $\sigma$ goes to zero. Consider for instance the case where $b(x) + \int_{\R^d} F(x, y) \nu^\sigma_\infty(\d y) = -\nabla V^\sigma (x)$ for some function $V^\sigma: \R^d \rightarrow \R$. Assume that $V^\sigma$ is a uniformly convex function, uniformly in $\sigma \in [0, \sigma_0]$:
\[ \exists \sigma_0, \eta_0 > 0: \forall \sigma \text{ with } \norm{\sigma}_2 \leq \sigma_0, \forall x \in \R^d, \quad  \nabla^2 V^\sigma(x) \geq \eta_0 I_d, \]
where $I_d$ denoted the identity matrix on $M_d(\R)$.
Then~\eqref{eq:W1 bound eberle} holds with constants independent of $\sigma$. In that case, the rate of convergence in $W_1$ norm in Theorem~\ref{th:main result} does not vanish as the noise vanishes. 
\subsubsection*{A simple explicit example}
We close this Section with a simple explicit example. Consider for $J \in \mathbb{R}^*$ the following McKean-Vlasov SDE on $\mathbb{R}$:
\begin{equation}
        \label{eq:toy MK cos}
        \d X_t = - X_t \d t + J \E \cos(X_t) \d t + \sqrt{2} \d B_t.
\end{equation}
The associated linear process $(Y^\alpha_t)$ is the solution of the Ornstein–Uhlenbeck SDE
\[ \d Y^\alpha_t = - Y^\alpha_t \d t + \alpha \d t + \sqrt{2} \d B_t.\]
This linear process admits a unique invariant probability measure given by
$\nu^\alpha_\infty = \mathcal{N}(\alpha, 1)$,
such that if $G$ is a standard Gaussian random variable, $\E \cos(Y^{\alpha, \nu^\alpha_\infty}_t) = \E \cos(\alpha
+ G) =  \frac{\cos(\alpha)}{\sqrt{e}}$.
We deduce that the invariant probability measures of~\eqref{eq:toy MK cos} are $\{ \mathcal{N}(\alpha, 1)~|~
\alpha \in \mathbb{R}, ~\frac{\sqrt{e}}{J} \alpha = \cos(\alpha)\}$.
Let $\alpha \in \mathbb{R}$ such that $\frac{\sqrt{e}}{J} \alpha = \cos(\alpha)$.
We have:
\[
\forall t \geq 0, \quad \Theta_t =  J \int_{\mathbb{R}}{ \frac{d}{dy} \E_y \cos(
Y^\alpha_t)  \nu_\infty^\alpha(\d y)} = -
\frac{J}{\sqrt{e}} e^{-t}
\sin(\alpha).
\]
So, for $\Re(z) > -1$, $\widehat{\Theta}(z) = - \frac{J}{z+1} e^{-1/2} \sin(\alpha)$
 and the equation
$\widehat{\Theta}(z) = 1$ has a unique solution $z = -J e^{-1/2} \sin(\alpha) -1$.
This root is strictly negative if and only if $J \sin(\alpha) > -\sqrt{e}$.
We deduce by Theorem~\ref{th:main result} that $\nu^\alpha_\infty$ is locally stable
 provided that $J \sin(\alpha) >
-\sqrt{e}$.
Recall that $\alpha \sqrt{e} = J \cos(\alpha)$. So  among all the invariant probability measures of~\eqref{eq:toy MK cos}, the (locally) stable ones are the $\mathcal{N}(\alpha, 1)$ with
\[  \alpha \tan(\alpha) > -1.  \]
\subsection{Notations}
For $x \in \R^d$, we denote by $|x|$ its Euclidean norm. Recall that $\cH := L^2(\nu_\infty)$ is the Hilbert space of measurable functions $h:\mathbb{R}^d \rightarrow \mathbb{R}^d$
such that \[ \norm{h}^2_{\cH} := \int{|h(x)|^2 \nu_\infty(\d x)} < \infty. \]
We denote by $\cK := W^{1, \infty}(\R^d; \R^d)$ the subspace of $\cH$ consisting of  all bounded and Lipschitz continuous functions $k: \R^d \rightarrow \R^d$. We equip $\cK$ with:
\[ \forall k \in \cK, \quad \norm{k}_{\cK} := \norm{k}_\infty + \norm{\nabla k}_\infty. \]
Note that, using Assumptions~\ref{ass:standing assumptions1} and~\ref{ass:standing assumption2}, the operators $\Theta_t$ and $\Omega_t$, defined by~\eqref{eq: definition of Theta t i j} and~\eqref{eq:volterra equation Theta Omega}, map $\cH$ to $\cK$.
Given $I$ a closed interval of $\mathbb{R}_+$, we denote by $C(I; \cK)$  the space of
continuous functions from $I$ to $\cK$.
Let $\alpha: \R^d \rightarrow \R^d$ satisfying~\eqref{eq:definition of alpha}. 
Let $k \in C(\mathbb{R}_+; \cK)$.
Consider $Y^{\alpha+ k, \nu}_{t}$ the solution of the following linear non-homogeneous $\mathbb{R}^d$-valued SDE
\begin{equation}
        \label{eq:linear non-homogeneous SDE}
	\d Y^{\alpha+k, \nu}_{t} = b(Y^{\alpha+k, \nu}_{t})\d t + \alpha(Y^{\alpha+k, \nu}_{t}) \d t + k_t(Y^{\alpha+k, \nu}_{t})\d t + \sigma \d B_t,
\end{equation}
where the initial condition satisfies $\Law(Y^{\alpha+k, \nu}_0) = \nu$.
Note that $Y^{\alpha+k, \nu}_{t}$ is a solution of~\eqref{eq:McKeanVlasov1} provided that $k$ satisfies the following
closure equation:
\begin{equation}
        \label{eq:closure equation}
\forall x \in \R^d, \quad \forall t \geq 0, \quad \alpha(x) + k_t(x) = \E F(x, Y^{\alpha+k, \nu}_{t}).
\end{equation}
Finally, for $y \in \mathbb{R}^d$ and $g$ a test function, we write $\E_y g(Y^{\alpha+k}_{t}) := \E g(Y^{\alpha+k, \delta_y}_{t})$.

A key ingredient in the proof of Theorem~\ref{th:main result} is the notion of Lions derivatives. 
A function $u: \cP_2(\R^d) \rightarrow \R$ is Lions differentiable at $\nu_0 \in \cP_2(\R^d)$ if there exists a deterministic function $\partial_\nu u(\nu_0): \R^d \rightarrow \R^d$ such that for every random variables $Z_0, H$ defined on the same probabilistic space $(\Omega, \cF, \P)$ with $\Law(Z_0) = \nu_0$ and $\E |H|^2 < \infty$, we have for $\nu = \Law(Z_0 + H)$:
\[ u(\nu) = u(\nu_0) + \E \left[ \partial_\nu u(\mu_0)(Z_0) \cdot H \right] + o(\sqrt{\E |H|^2}),  \quad \text{ as $\E |H|^2$ goes to zero}.  \]
We refer to~\cite{MR3752669} for a detailed  presentation of this theory. Let $\psi \in C^1(\R^d)$ with $\norm{\nabla \psi}_\infty < \infty$. We recall that the linear functional
\[ u(\nu) = \int_{\R^d} \psi(y) \nu(\d y) \]
is Lions differentiable at every $\nu_0 \in \cP_2(\R^d)$, with a Lions derivative given by
\[ \partial_\mu u(\nu_0)(y) = \nabla \psi (y). \]
\subsection{An integrated sensibility formula}\label{sec:integrated sensibility formula}
Let $\alpha: \R^d \rightarrow \R^d$ satisfying~\eqref{eq:definition of alpha}.
The goal of this Section is to prove the following ``integrated sensitivity  formula''.
\begin{proposition}\label{prop:integrated sensitivity formula}
	Let $k \in C(\mathbb{R}_+; \cK)$ and $\nu \in \mathcal{P}_1(\mathbb{R}^d)$. 
	Let $g \in C^2(\mathbb{R}^d)$ with $\norm{\nabla g}_\infty+ \norm{\nabla^2 g}_\infty < \infty$. It holds that for all 
	$t \geq 0$,
\[
	\E g(Y^{\alpha+k, \nu}_{t}) - \E g(Y^{\alpha, \nu}_{t}) = \int_0^t{ \int_{\mathbb{R}^d}{ \left[ \nabla_y \E_y g(Y^\alpha_{t - \theta}) \cdot 
			k_\theta(y) \right] \Law(Y^{\alpha+k, 
	\nu}_\theta)(\d y)} \d \theta}. 
\]
\end{proposition}
Without loss of generality, we can assume that the initial condition $\nu$ belongs to  $\cP_2(\R^d)$. We define for all $0 \leq \theta \leq t$:
\[ \forall \nu \in \cP_2(\R^d), \quad  u^{g}_{t, \theta}(\nu) := \E g(Y^{\alpha, \nu}_{t - \theta}). \]
By the Markov property, $u^g_{t, \theta}$ is linear with respect to $\nu$:
\[ u^g_{t, \theta}(\nu) = \int_{\R^d} \E_y g(Y^{\alpha}_{t - \theta}) \nu(\d y). \]
By~\cite[Th. 7.18]{MR3097957}, the function $y \mapsto \E_y g(Y^{\alpha}_{t - \theta})$ is continuously differentiable with a bounded derivative.
So $u^g_{t, \theta}$ is Lions differentiable with
\[ \partial_\nu u^g_{t, \theta}(\nu)(y)  = \nabla_y \E_y g(Y^\alpha_{t - \theta}). \]
Therefore, it suffices to prove that
\[  	\E g(Y^{\alpha+k, \nu}_{t}) - \E g(Y^{\alpha, \nu}_{t}) =  \int_0^t{  \E \left[  \partial_\nu 
	u^{g}_{t, 
\theta}(\Law(Y^{\alpha+k, \nu}_{\theta}))(Y^{\alpha+k, \nu}_{\theta}) \cdot k_\theta(Y^{\alpha+k, \nu}_{\theta}) \right]  
	\d \theta}.  \]
Given $\theta \geq 0$ and $k \in C(\mathbb{R}_+; \cK)$, we write for all $u \geq 0$:
	\begin{equation} 
	\label{eq:def de h theta}
	\forall y \in \R^d, \quad k^{[\theta]}_u(y) := \begin{cases} k_u(y) \quad \text{ if } u \leq \theta \\ 0 \quad \text { if } u > \theta \end{cases}	\end{equation}
The proof of Proposition~\ref{prop:integrated sensitivity formula} is deduced from the following Lemma and from the fundamental theorem of calculus.
\begin{lemma}\label{lem:derivative theta h}
	The function $\theta \mapsto \E g(Y^{\alpha+k^{[\theta]}, \nu}_{t})$
	is differentiable for all $\theta \in (0, t)$ and
	\begin{equation}
		\label{eq:derivative h theta}
		\frac{d}{d \theta}  \E g(Y^{\alpha+k^{[\theta]}, \nu}_{t}) = \E \partial_\nu u^{g}_{t, 
			\theta}(\Law(Y^{\alpha+k, \nu}_{\theta}))(Y^{\alpha+k, \nu}_{\theta}) \cdot k_\theta(Y^{\alpha+k, \nu}_{\theta}).  
	\end{equation}
\end{lemma}
\input{fig}
\begin{proof}
	Fix $\theta \in (0, t)$ and 
	$\delta > 0$ small enough such that $\theta+\delta \in (0, t)$. We write
\[	\begin{array}{ccc}
		Y_\theta := Y^{\alpha+k, \nu}_{\theta}, &\quad Y^1_{\theta+\delta} := Y^{\alpha+k^{[\theta]}, 
			\nu}_{\theta+\delta},&\quad Y^2_{\theta+\delta} := Y^{\alpha+k, \nu}_{\theta+\delta} \\
	\mu_\theta := \Law(Y_{\theta}), &\quad \mu^1_{\theta+\delta} := 
	\Law(Y^1_{\theta+\delta}), &\quad  \mu^2_{\theta+\delta} := \Law(Y^2_{\theta+\delta}).
	\end{array}
\]
	The notations are illustrated on Figure~\ref{fig:intregated_formula_illustrated}. We have by the Markov property 
	satisfied by $Y$ at time $\theta+\delta$
	\[ \E g(Y^{\alpha+k^{[\theta+\delta]}, \nu}_{t}) - \E g(Y^{\alpha+k^{[\theta]}, \nu}_{t}) = \E g(Y^{\alpha, 
	\mu^2_{\theta+\delta}}_{t -  
(\theta+\delta)}) -  
\E g(Y^{\alpha, \mu^1_{\theta+\delta}}_{t - (\theta+\delta)}).  \]
	By definition of the Lions derivative at the point $\mu^1_{\theta+\delta}$ we have
	\begin{equation}
		\label{eq:lions derivate in mu1}
		\E g(Y^{\alpha, \mu^2_{\theta+\delta}}_{t - (\theta+\delta)}) -  
		\E g(Y^{\alpha, \mu^1_{\theta+\delta}}_{t - (\theta+\delta)})=  \E \partial_\nu u^{g}_{t, 
			\theta+\delta}(\mu^1_{\theta+\delta})(Y^1_{\theta + \delta}) \cdot 
			(Y^2_{\theta+\delta} - Y^1_{\theta+\delta}) + o( \sqrt{\E \abs{Y^2_{\theta+\delta} - Y^1_{\theta+\delta}}^2} ).
	\end{equation}
By Lemma~\ref{lem:technical estimates SDE}~(a)  below, it holds that $o( \sqrt{\E \abs{Y^2_{\theta+\delta} - Y^1_{\theta+\delta}}^2}) = o(\delta)$ as $\delta$ goes to zero.
	We now approximate $Y^1_{\theta+\delta}$ and $Y^2_{\theta+\delta}$ by a one-step Euler scheme:
	\begin{align}
		\label{eq:euler scheme}
		Y^1_{\theta+\delta} &\approx \tilde{Y}^1_{\theta + \delta} := Y_\theta + \alpha(Y_\theta) 
		\delta + \sigma (W_{\theta+\delta} - W_\theta) \\
		Y^2_{\theta+\delta} &\approx \tilde{Y}^2_{\theta + \delta} :=  Y_\theta + \alpha(Y_\theta) \delta
		+ k_\theta(Y_\theta) \delta + \sigma  (W_{\theta+\delta} - W_\theta)\nonumber
	\end{align}
	Note that $\tilde{Y}^2_{\theta + \delta} -  \tilde{Y}^1_{\theta + \delta} = k_\theta(Y_\theta) \delta$.
	These one-step Euler schemes have an error in $L^2$ norm of size $o(\delta)$ (see Lemma~\ref{lem:technical estimates SDE}~(c) below), therefore:
	\[ \sqrt{ \E \abs{Y^1_{\theta+\delta} - \tilde{Y}^1_{\theta+\delta}}^2}  + \sqrt{ \E \abs{Y^2_{\theta+\delta} - \tilde{Y}^2_{\theta+\delta}}^2} = o(\delta), \]
so~\eqref{eq:lions derivate in mu1} gives
	\[	\E g(Y^{\alpha, \mu^2_{\theta+\delta}}_{t - (\theta+\delta)}) -  
	\E g(Y^{\alpha, \mu^1_{\theta+\delta}}_{t -  (\theta+\delta)}) =  \delta \E \left[ \partial_\nu u^{g}_{t, 
\theta+\delta}(\mu^1_{\theta+\delta})(Y^1_{\theta + \delta}) \cdot k_\theta(Y_\theta) \right] + o(\delta). \]
	Finally, one has
	\begin{align*}
		& \left|  \E \left[ \partial_\nu u^{g}_{t, 
		\theta+\delta}(\mu^1_{\theta+\delta})(Y^1_{\theta + \delta}) \cdot k_\theta(Y_\theta) \right] -  \E \left[ \partial_\nu 
		u^{g}_{t, 
	\theta}(\mu_\theta)(Y_\theta) \cdot k_\theta(Y_\theta) \right] \right| \\
		& \quad \leq \left|\E \left[ \partial_\nu u^{g}_{t, 
		\theta+\delta}(\mu^1_{\theta+\delta})(Y^1_{\theta + \delta}) \cdot k_\theta(Y_\theta) \right] - \E \left[ \partial_\nu 
		u^{g}_{t, 
	\theta + \delta}(\mu_\theta)(Y_\theta) \cdot k_\theta(Y_\theta) \right] \right| \\
		& \quad \quad + \left| \E \left[ \partial_\nu u^{g}_{t, \theta+\delta}(\mu_\theta)(Y_\theta) \cdot 
		k_\theta(Y_\theta) - \E \partial_\nu u^{g}_{t, \theta}(\mu_\theta)(Y_\theta) \cdot k_\theta(Y_\theta) \right] \right| 
		=: 
		A_1 + A_2.
	\end{align*}
By Lemma~\ref{ass:regualirty L derivative}~(a) there exists a constant $C(t)$ such that
	\[ A_1 \leq  C(t)  \norm{k_\theta}_{\cK} \sqrt{ \E |Y^1_{\theta+\delta} - Y_\theta |^2 } 
		\overset{\text{Lem.}~\ref{lem:technical estimates SDE}~(b)}{\leq} C(t) \sqrt{\delta} \sup_{\theta \in [0, 
t]}\norm{k_\theta}_{\cK}.   \]
	Let $\epsilon > 0$ be fixed. Lemma~\ref{ass:regualirty L derivative}~(b) yields for $\delta$ small enough:
	\[ A_2 \leq \epsilon \sup_{\theta \in [0, t]}\norm{k_\theta}_{\cK}. \]
	Altogether, we find that
	\[ \E g(Y^{a+k^{[\theta+\delta]}, \nu}_{t}) - \E g(Y^{a+k^{[\theta]}, \nu}_{t})  = \delta \E \partial_\nu u^{g}_{t, 
		\theta}(\mu_\theta)(Y_\theta) \cdot k_\theta(Y_\theta) + o(\delta). \]
	This ends the proof.
\end{proof}	
We used the following classical estimates (the constants depend on $\alpha$ and $k$):
\begin{lemma}\label{lem:technical estimates SDE}
	We have, with the notations introduced in the proof of Lemma~\ref{lem:derivative theta h}, 
	\begin{enumerate}[label=~(\alph*), leftmargin=*]
		\item it holds that $\E |Y^{2}_{\theta+\delta} - 
			Y^{1}_{\theta+\delta}|^2 \leq C(t) \delta^2$.
		\item it holds that $\E \left|Y^1_{\theta+\delta} - Y_\theta \right|^2 \leq C(t) \delta$.
		\item the Euler scheme~\eqref{eq:euler scheme} satisfies
		$\E |Y^1_{\theta+\delta} - \tilde{Y}^1_{\theta+\delta}|^2 + \E |Y^2_{\theta+\delta} - 
		\tilde{Y}^2_{\theta+\delta}|^2 = 
		o(\delta^2)$,
		as $\delta$ goes to zero.
	\end{enumerate}
\end{lemma}
We also used the following regularity results on $\partial_\nu u^{g}_{t, s}(\nu)(y) = \nabla_y \E_y g(Y^{\alpha}_{t - s})$. The proofs follow easily from the stochastic representation of $y \mapsto \nabla_y \E_y g(Y^\alpha_{t - s})$: in particular this function has a bounded derivative (because $\norm{\nabla g}_\infty + \norm{\nabla^2 g}_\infty < \infty$, see~\cite[Th. 7.18]{MR3097957}).  
\begin{lemma}\label{ass:regualirty L derivative}
	It holds that:
	\begin{enumerate}[label=~(\alph*), leftmargin=*]
		\item there exists a constant 
		$C(t)$ such that any square-integrable variables $Z, Z'$, 
		\[\sup_{0 \leq s \leq t }   \E \left|  \partial_\nu u^{g}_{t,s}(\Law(Z))(Z) -  
		\partial_\nu u^{g}_{t,s}(\Law(Z'))(Z') 
		\right|^2  
		\leq C(t) \E |Z-Z'|^2.
		\]
		\item the function $s \mapsto \partial_\nu u^{g}_{t,s}(\Law(Z))(Z)$ is continuous: for all 
		$\epsilon > 0$
		there exists $\delta > 0$:
		\[ \forall s, s' \in [0, t], \quad |s-s'| \leq \delta \implies \E \left|\partial_\nu u^{g}_{t,s'}(\Law(Z))(Z) - 
		\partial_\nu u^{g}_{t,s}(\Law(Z))(Z) \right| ^2 < 
		\epsilon.\]
	\end{enumerate}
\end{lemma}
\begin{remark}
	It is also possible to prove Proposition~\ref{prop:integrated sensitivity formula} without using Lions derivatives. 
	Fix $t > 0$ and define for $s \in (0, t)$
	\[ (s, y) \mapsto   \phi(s, y) :=  \E_y g(Y^\alpha_{t-s}). \]
	Given $k \in C([0, t]; \cK)$, we denote by
	\[  \mathcal{L}^{\alpha+k}_\theta \psi := (b+ \alpha + k_\theta) \cdot \nabla \psi + \frac{1}{2} \sum_{i,j = 1}^d {(\sigma \sigma^*)}_{i,j} \partial_{x_i} \partial_{x_j} \psi, \quad \theta \in [0,t],  \] the infinitesimal 
	generator associated to $Y^{\alpha+k}$.
	It holds that $\phi \in C^{1,2}([0,t) \times \mathbb{R}^d)$ with 
	\[ \frac{\partial}{\partial s} \phi(s, y) = - \mathcal{L}_s^\alpha \phi(s, y). \]
	So, by Itô's lemma, 
	\begin{align*}
		\E \phi(s, Y^{\alpha+k, \nu}_s)  &= \E \phi(0, Y^{\alpha+k, \nu}_0) - \int_0^s{ \E 
			\mathcal{L}^\alpha_\theta \phi(\theta, Y^{\alpha+k, \nu}_\theta) \d \theta} 
		+ 
		\int_0^s{ \E \mathcal{L}^{\alpha+k}_\theta \phi(\theta, Y^{\alpha+k, \nu}_\theta) \d \theta} \\
		&= \E \phi(0, Y^{\alpha+k, \nu}_0) + \int_0^s{ \E \nabla_y \phi(\theta, 
		Y^{\alpha+k, \nu}_\theta) \cdot  k_\theta(Y^{\alpha+k, \nu}_\theta) \d \theta}.
	\end{align*}
	We used that $\mathcal{L}^{\alpha+k}_\theta \psi - \mathcal{L}^{\alpha}_\theta \psi  = \nabla \psi \cdot k_\theta$. Using the 
	definition of 
	$\phi$, we find:
	\[ 	\E \phi(s, Y^{\alpha+k, \nu}_s) = \E g(Y^{\alpha, \nu}_t) + \int_0^s{ \int_{\mathbb{R}^d}{ \nabla_y\E_y 
			g(Y^\alpha_{t -\theta}) \cdot k_\theta(y)   
			 \mathcal{L}(Y^{\alpha+k, \nu}_\theta)(\d y) \d \theta}}. \]
	Finally, we let $s$ converges to $t$ and find the stated formula.
\end{remark}
As a corollary, we obtain the following apriori control of the Wasserstein distance between two solutions of SDE driven by slightly different drifts. Recall that $C_*, \kappa_* > 0$ are given by~\eqref{eq:W1 bound eberle}.
\begin{corollary}\label{cor:wassertein control with h}
For all $t \geq 0$ and $k \in C([0,t]; \cK)$, it holds that
	\[  W_1(\Law(Y^{\alpha+k, \nu}_t), \Law(Y^{\alpha, \nu}_t)) \leq C \int_0^t{ e^{-\kappa_*(t-\theta)} \norm{k_\theta}_\infty \d \theta }. \]
\end{corollary}
\begin{proof}
	Let $g \in C^2(\R^d)$. Using~\eqref{eq:gradient bound} and Proposition~\ref{prop:integrated sensitivity formula}, we have for $g \in C^2(\R^d)$, 
	\[ \abs{\E g(Y^{\alpha + k, \nu}_t) - \E g(Y^{\alpha, \nu}_t)} \leq C_* \norm{\nabla g}_\infty \int_0^t e^{-\kappa_*(t-\theta)} \norm{k_\theta}_\infty  \d \theta. \]
	This inequality also holds if $g \in C^1(\R^d)$ by a standard approximation argument.
\end{proof}
Note that by choosing $g = F(x, \cdot)$, Proposition~\ref{prop:integrated sensitivity formula} gives:
\begin{equation}\label{eq:representation formula}
\E F(x, Y^{\alpha+k, \nu}_t) - \E F(x, Y^{\alpha,\nu}_t) = \int_0^t{ \int_{\R^d} \nabla_y \E_y F(x, Y^{\alpha}_{t-\theta}) \cdot k_\theta(y) \Law(Y^{\alpha+k, \nu}_\theta)(\d y) \d \theta}.
\end{equation}
Recall that $\Theta_t(k)$ is defined by~\eqref{eq: definition of Theta t i j}.
When $\nu = \nu_\infty$ and when $k$ is small, we obtain:
\[ \E F(x, Y^{\alpha+k, \nu_\infty}_t) - \alpha(x) \approx \int_0^t{ \int_{\R^d} \nabla_y \E_y F(x, Y^\alpha_{t-\theta}) \cdot k_\theta(y) \nu_\infty(\d y) \d \theta} = \int_0^t{ \Theta_{t-\theta}(k_\theta)(x) \d \theta}.
\]
This observation is crucially used in the next Section.
\subsection{Control of the non-linear interactions}
Recall that $(X^\nu_t)$ denotes the solution of the McKean-Vlasov equation~\eqref{eq:McKeanVlasov1}. 
For all $t \geq 0$ and $\nu \in \cP_1(\R^d)$, we define
\begin{align*} 
	\varphi^\nu_t(x) &:= \E F(x, Y^{\alpha, \nu}_t) - \alpha(x), \\
	k^\nu_t(x) &:= \E F(x, X^\nu_t) - \alpha(x).
\end{align*}
Recall that $\Omega_t(h)$ is defined by~\eqref{eq:link between Omega and Theta}.
In this Section, we prove that:
\begin{proposition}\label{prop:control hnu phinu Omeganu}
		For all $T > 0$, there is a constant $C_T$ such that for all $t \in [0, T]$, for all  $x \in \R^d$ and for all $\nu \in \cP_1(\R^d)$:
		\[ \left|k^\nu_t(x) - \varphi^\nu_t(x) - \int_0^t{\Omega_{t-s} (\varphi^\nu_s)(x) \d s} \right| \leq  C_T{(W_1(\nu, \nu_\infty))}^2. \]
\end{proposition}
The first step is to show the following apriori estimate on $k^\nu_t$:
\begin{lemma}\label{lem:technical estimate h nu}
	Let $T > 0$. There exists a constant $C_T$ such that
	\[ \forall \nu \in \cP_1(\R^d), \quad \sup_{t \in [0, T]} \norm{k^\nu_t}_{\cK} \leq C_T W_1(\nu, \nu_\infty). \]
	In addition,  $k^\nu \in C([0, T]; \cK)$.
\end{lemma}
\begin{proof}We have
\[ k^\nu_t(x) = \int_{\R^d} F(x,y) (\Law(X^\nu_t)- \nu_\infty)(\d y). \]
Using Lemma~\ref{lem:Lipschitz McKean} below, we deduce that
$|k^\nu_t(x)| \leq C_T \norm{ \nabla_y F}_\infty W_1(\nu, \nu_\infty)$.
Similarly,
\[ |\nabla k^\nu_t(x)| \leq C_T \norm{ \nabla^2_{x,y} F}_\infty W_1(\nu, \nu_\infty).\] This shows the bound on $\sup_{t \in [0, T]} \norm{k^\nu_t}_{\cK}$. Moreover,
\[ k^\nu_t(x) - k^\nu_s(x) = \int_{\R^d} F(x, y) \left( \Law(X^\nu_t) - \Law(X^\nu_s)  \right)(\d y). \]
In addition, for all $T > 0$ and $\nu \in \cP_1(\R^d)$, there exists a constant $C(T, \nu)$ such that
\[ \forall 0 \leq s \leq t \leq T, \quad  W_1(\Law(X^\nu_t), \Law(X^\nu_s)) \leq \E \left| X^\nu_t - X^\nu_s \right|\leq C(T, \nu) \sqrt{t-s}. \]
We deduce that $k^\nu \in C([0, T]; \cH)$. 
\end{proof}
Next, using Proposition~\ref{prop:integrated sensitivity formula}, we show that:
\begin{lemma}\label{lem:technical estimate h nu phi nu}
	For all $T > 0$, there is a constant $C_T$ such that for all $t \in [0, T]$, for all  $x \in \R^d$ and for all $\nu \in \cP_1(\R^d)$:
	\[ \left|k^\nu_t(x) - \varphi^\nu_t(x) - \int_0^t{\Theta_{t-s} (k^\nu_s)(x) \d s} \right| \leq  C_T{(W_1(\nu, \nu_\infty))}^2. \]
\end{lemma}
\begin{proof}
Using the closure equation~\eqref{eq:closure equation}, we have:
\begin{align*} k^\nu_t(x) &= \E F(x, Y^{\alpha+k^\nu, \nu}_t) - \alpha(x) \\
	&=\E F(x, Y^{\alpha+k^\nu, \nu}_t) - \E F(x, Y^{\alpha, \nu}_t)  + \varphi^\nu_t(x).  
\end{align*}
We apply Proposition~\ref{prop:integrated sensitivity formula} and obtain:
\[ k^\nu_t(x) - \varphi^\nu_t(x) = \int_0^t{ \int_{\R^d} \nabla_y \E_y F(x, Y^{\alpha}_{t-\theta}) \cdot k^\nu_\theta(y) \Law(X^{\nu}_\theta)(\d y) \d \theta}. \]
We used that $\Law(Y^{\alpha+k^\nu, \nu}_\theta) = \Law(X^\nu_\theta)$.
Let 
\[ G^x_{t, \theta}(y) := \nabla_y \E_y F(x, Y^{\alpha}_{t-\theta}) \cdot k^\nu_\theta(y).   \]
We deduce that:
\[ k^\nu_t(x) - \varphi^\nu_t(x) = \int_0^t{ \Theta_{t-\theta}(k^\nu_s)(x) \d s } + R_t(x),\]
where
\[ R_t(x) := \int_0^t{ \left( \E G^x_{t, \theta}(X^\nu_\theta) - \E G^x_{t, \theta}(X^{\nu_\infty}_\theta)  \right)\d \theta. }\]
Using Lemma~\ref{lem:technical estimate h nu}, we deduce for all $T > 0$, there exists a constant $C_T$, such that
\[ \forall x \in \R^d, \forall 0 \leq \theta \leq t \leq T, \quad \norm{\nabla G^x_{t, \theta}}_\infty \leq C_T W_1(\nu, \nu_\infty). \]
Therefore, using Lemma~\ref{lem:Lipschitz McKean} below, we deduce that $\left| R_t(x) \right| \leq C_T {(W_1(\nu, \nu_\infty))}^2$. 
\end{proof}
By iterating the estimate of Lemma~\ref{lem:technical estimate h nu phi nu}, we deduce the proof of Proposition~\ref{prop:control hnu phinu Omeganu}.
\begin{proof}[Proof of Proposition~\ref{prop:control hnu phinu Omeganu}]
	For $i \in \N$, define: 
	\[ \psi^i_t(x) := \varphi^\nu_t(x) + \int_0^t{ \Theta_{t-\theta} (\varphi^\nu_\theta)(x) \d \theta} + \cdots  + \int_0^t{ \Theta^{\otimes i}_{t - \theta}(\varphi^\nu_\theta)(x) \d \theta }. \]
	Let $D_T$ a constant such that for all $k \in \cK$, $\sup_{t \in [0, T]} \norm{\Theta_t(k)}_\infty \leq D_T \norm{k}_\infty$. By induction, we have:
	\[ \forall t \in [0, T], \quad \norm{\Theta^{\otimes i}_t(k)}_\infty \leq {(D_T)}^i \frac{t^{i-1}}{(i-1)!} \norm{k}_\infty. \]
Let $C_T$ be the constant of Lemma~\ref{lem:technical estimate h nu phi nu}. By induction, we deduce that:
\[ \norm{ k^\nu_t - \psi^i_t - \int_0^t{ \Theta^{\otimes (i+1)}_{t-\theta}(k^\nu_\theta) \d \theta} }_\infty \leq C_T {(W_1(\nu, \nu_\infty))}^2 \left( 1 + D_T t + \cdots + \frac{{(D_T t)}^i}{i!} \right). \]
	To conclude, it suffices to take the limit $i \rightarrow \infty$: $\psi^i_t$ converges uniformly to $\varphi^\nu_t + \int_0^t{\Omega_{t-\theta}(\varphi^\nu_s) \d s}$, and $\int_0^t{ \Theta^{\otimes (i+1)}_{t-\theta}(k^\nu_\theta) \d \theta}$ converges uniformly to zero. This ends the proof. 
\end{proof}
We used the following apriori estimate on the solution of~\eqref{eq:McKeanVlasov1}.
\begin{lemma}\label{lem:Lipschitz McKean}
	Let $T > 0$. There exists a constant $C_T$ such that for all $\mu_1, \mu_2 \in \cP_1(\R^d)$, 
	\[ \forall t \in [0, T], \quad W_1(\Law(X^{\mu_1}_t), \Law(X^{\mu_2}_t)) \leq C_T W_1(\mu_1, \mu_2). \]
\end{lemma}
\begin{proof}
	Consider $(X^{\mu_1}_t, X^{\mu_2}_t)$ the solutions of~\eqref{eq:McKeanVlasov1} coupled with the same 
	Brownian motion. The initial conditions $(X^{\mu_1}_0, X^{\mu_2}_0)$ are chosen such that 
	$\E |X^{\mu_1}_0 - X^{\mu_2}_0| = W_1(\mu_1, \mu_2)$. Let $\mu^1_t := \Law(X^{\mu_1}_t)$ and $\mu^2_t := \Law(X^{\mu_2}_t)$. 
	From~\eqref{eq:McKeanVlasov1} and Assumption~\ref{ass:standing assumptions1}, we have
	\begin{align*} \E \left| X^{\mu_1}_t - X^{\mu_2}_t \right| & \leq \E \left| X^{\mu_1}_0 - X^{\mu_2}_0 \right| 
		 +  \E \int_0^t \left|  b(X^{\mu_1}_s) - b(X^{\mu_2}_s ) \right| \d s \\
 	& \quad + \int_0^t \int_{\R^d} \left| \E F(X^{\mu_1}_s, y) - \E F(X^{\mu_2}_s, y)\right| \Law(X^{\mu_1}_s)(\d y) \d s \\
	& \quad + \int_0^t \left| \int_{\R^d} \E F(X^{\mu_2}_s, y) (\Law(X^{\mu_1}_s) - \Law(X^{\mu_2}_s))(\d y) \right| \d s.
	\end{align*}
	The functions $F$ and $b$ are Lipschitz, so there exists a constant $L$ such that 
	\[ \E \left| X^{\mu_1}_t - X^{\mu_2}_t \right| \leq  \E \left| X^{\mu_1}_0 - X^{\mu_2}_0 \right|  + L \int_0^t{  \E \left| X^{\mu_1}_s - X^{\mu_2}_s\right| \d s}. \]
	By Grönwall's inequality, we deduce that
	\[ W_1(\Law(X^{\mu_1}_t), \Law(X^{\mu_2}_t)) \leq \E |X^{\mu_1}_t - X^{\mu_2}_t| \leq e^{L t} \E |X^{\mu_1}_0 - X^{\mu_2}_0| = e^{Lt} W_1(\mu_1, \mu_2). \]
\end{proof}
\subsection{Proof of Theorem~\ref{th:main result}}\label{sec:proof of main result 1}
We now give the proof of Theorem~\ref{th:main result}.
By combining Proposition~\ref{prop:control hnu phinu Omeganu}, Corollary~\ref{cor:wassertein control with h}, and Proposition~\ref{prop:paley-wiener}, we obtain:
\begin{lemma}\label{lem:final lemma}
	Assume that $\lambda' > 0$, where $\lambda'$ is given by~\eqref{eq:spectral assumption}.
	Let $\lambda \in (0, \lambda')$. There exists a constant $C_\lambda$ such that for all $T > 0$, there is a constant $C_T$ such that for all $\nu \in \cP_1(\R^d)$ and for all $t \in [0, T]$:
	\[ W_1(\Law(X^\nu_t), \nu_\infty) \leq C_\lambda e^{-\lambda t} W_1(\nu,\nu_\infty) + C_T{(W_1(\nu,\nu_\infty))}^2.  \]
\end{lemma}
\noindent Importantly, the constant $C_\lambda$ above does not depend on $T$.
\begin{proof}
	We write
	\begin{align*}
		W_1(\Law(X^\nu_t), \nu_\infty) &=  W_1(\Law(Y^{\alpha+k^\nu, \nu}_t), \nu_\infty) \\
					      & \leq  W_1(\Law(Y^{\alpha+k^\nu, \nu}_t), \Law(Y^{\alpha, \nu}_t)) + W_1(\Law(Y^{\alpha, \nu}_t), \nu_\infty) \\
					      & \leq~ C_* \int_0^t{ e^{-\kappa_*(t-\theta)} \norm{k^\nu_\theta}_\infty \d \theta} + C_* W_1(\nu, \nu_\infty) e^{-\kappa_* t}. 
	\end{align*}
		We used Corollary~\ref{cor:wassertein control with h} to estimate $W_1(\Law(Y^{\alpha+k^\nu, \nu}_t), \Law(Y^{\alpha, \nu}_t))$  and~\eqref{eq:W1 bound eberle} as well as Markov property to estimate $W_1(\Law(Y^{\alpha, \nu}_t), \nu_\infty)$.
	Applying Proposition~\ref{prop:control hnu phinu Omeganu}, we deduce that
	\begin{align*} \int_0^t{ e^{-\kappa_*(t-\theta)} \norm{k^\nu_\theta}_\infty \d \theta} & \leq \int_0^t{ e^{-\kappa_*(t-\theta)} \left[ \norm{\varphi^\nu_\theta}_\infty + \int_0^\theta{ \norm{\Omega_{\theta-u}(\varphi^\nu_u)}_\infty  \d u } \right] \d\theta }  + C_T {(W_1(\nu, \nu_\infty))}^2.
	\end{align*}
	The estimate~\eqref{eq:gradient bound} implies that $\norm{\varphi^\nu_\theta}_\infty \leq C_* W_1(\nu, \nu_\infty) \norm{\nabla_y F}_\infty e^{-\kappa_* \theta}$.
	Fix $\lambda \in (0, \lambda')$.
	In view of~\eqref{eq:control infty norm theta},~\eqref{eq:volterra equation Theta Omega} and the proof of Proposition~\ref{prop:paley-wiener}, the spectral assumption $\lambda' > 0$ implies that there exists a constant $C$ such that
	\[  \forall h \in \cH, \quad \norm{\Omega_t(h)}_\infty \leq C e^{-\lambda t} \norm{h}_{\cH}. \]
	Therefore, using that $\lambda < \kappa_*$, we have
	\[ \int_0^\theta \norm{\Omega_{\theta-u}(\varphi^\nu_u)}_\infty \d u \leq C \int_0^\theta e^{-\lambda(\theta-u)} e^{-\kappa_* u} W_1(\nu, \nu_\infty) \d u \leq C_\lambda e^{-\lambda \theta} W_1(\nu, \nu_\infty). \]
	Altogether, we deduce the stated inequality. 
\end{proof}
Finally, the proof of Theorem~\ref{th:main result} is deduced from Lemma~\ref{lem:final lemma} by following the argument of~\cite[Proposition 5.2]{zbMATH06380861}.
\begin{proof}[Proof of Theorem~\ref{th:main result}.]
We choose $T$ large enough such that $C_\lambda e^{-\lambda T} \leq \frac{1}{4}$. We choose $\epsilon > 0$ small enough such that 
\[ W_1(\nu, \nu_\infty) \leq \epsilon \implies C_T{(W_1(\nu, \nu_\infty))}^2 \leq \frac{1}{4} W_1(\nu, \nu_\infty). \]
Therefore we have, by induction, provided that $W_1(\nu, \nu_\infty) \leq \epsilon$:
\[ W_1(\Law(X^\nu_{kT}), \nu_\infty) \leq {(1/2)}^k W_1(\nu, \nu_\infty). \]
We write $t = kT  + s$ for some $s \in [0, T)$. 
Using Lemma~\ref{lem:Lipschitz McKean}, there exists a constant $C$ such that
\[  W_1(\Law(X^\nu_t), \nu_\infty) \leq C {(1/2)}^k W_1(\nu, \nu_\infty) \leq \frac{C}{2} e^{-c t} W_1(\nu, \nu_\infty),\]
where $c := \frac{\log(2)}{T}$. This ends the proof of Theorem~\ref{th:main result}.
\end{proof}
\subsection{Connections with Lions derivatives}\label{sec:lions derivatives}
In this Section, we give a probabilistic interpretation of the linear maps $\Theta_t(h)$ and $\Omega_t(h)$.
Let $(\Omega, \cF, \P)$ be a probability space.
Let $Z_0, H \in L^2(\Omega, \cF, \P)$ such that $\Law(Z_0) = \nu_\infty$. We denote by $h: \R^d \rightarrow \R^d$ a measurable function such that:
\[ \P(\d \omega) p.s. \quad \E [H~|~Z_0] = h(Z_0). \] 
It follows from $\E |H|^2 < \infty$ and from the Cauchy-Schwarz inequality that 
$h \in \cH$. Define for all $x \in \R^d$, $t \geq 0$ and $\nu \in \cP_2(\R^d)$:
\begin{align*}
	u^x_t(\nu) &:= \E F(x, Y^{\alpha, \nu}_t), \\
	v^x_t(\nu) &:= \E F(x, X^\nu_t). 
\end{align*}
\begin{proposition}Under Assumptions~\ref{ass:standing assumptions1} and~\ref{ass:standing assumption2}, we have:
	\begin{enumerate}[label=~(\alph*)]
		\item There exists a constant $C_T$ such that for all random variables $Z_0, H$ with $\Law(Z_0) = \nu_\infty$ and $\E |H|^2 < \infty$,  for all $t \in [0, T]$ and for all $x$ we have, for $\nu = \Law(Z_0 + H)$:
	\[ \left| u^x_t(\nu) - \alpha(x)- \Theta_t(h)(x) \right| \leq C_T \E |H|^2.  \]

		\item The function $u^x_t$ is Lions differentiable at $\nu_\infty$, with a derivative given by
			\[ \partial_\nu u^x_t(\nu_\infty)(y) = \nabla_y \E_y F(x, Y^{\alpha}_t). \]
	\item There exists a constant $C_T$ such that for all random variables $Z_0, H$ with $\Law(Z_0) = \nu_\infty$ and $\E |H|^2 < \infty$,  for all $t \in [0, T]$ and for all $x$ we have, for $\nu = \Law(Z_0 + H)$:
	\[ \left| v^x_t(\nu) - \alpha(x)- \Omega_t(h)(x) \right| \leq C_T \E |H|^2.  \]
\item The function $v^x_t$ is Lions differentiable at $\nu_\infty$, with a derivative given by
	\begin{equation}
		\label{eq:lions derivative McKean}
		\partial_\nu v^x_t(\nu_\infty)(y) = \nabla_y \E_y F(x,Y^\alpha_t) + \int_0^t{ \Omega_{t-s}( \nabla_y \E_y F(\cdot, Y^\alpha_s))(x) \d s }. 
	\end{equation}
			\end{enumerate}
\end{proposition}
\begin{remark}\label{rk:interpretation Theta_t, Omega_t}
	In particular, from points 1 and 3, we have for all $h \in \cH$:
	\begin{align*}
		\Theta_t(h)(x) &= \lim_{\epsilon \rightarrow 0} \frac{ u^x_t(\Law(Z_0 + \epsilon h(Z_0))) - \alpha(x)}{\epsilon}, \\
		\Omega_t(h)(x) &= \lim_{\epsilon \rightarrow 0} \frac{ v^x_t(\Law(Z_0 + \epsilon h(Z_0))) - \alpha(x)}{\epsilon}.
	\end{align*}
\end{remark}
\begin{proof}
	The first two points follow from the fact that $u^x_t(\nu)$ depends linearly on $\nu$: $u^x_t(\nu) = \E g(Z_0 + H)$, with $g(y) := \E_y F(x, Y^\alpha_t)$. This function $g$ is $C^2$ and there exists a constant $C_T$ such that for all $t \in [0, T]$, for all $x \in \R^d$, $\norm{\nabla^2 g}_\infty \leq C_T$. Therefore:
	\begin{align*} u^x_t(\nu) - \alpha(x) &= \E g(Z_0 + H) - \E g(Z_0) = \E \int_0^1{ \nabla g(Z_0 + \theta H)\cdot H \d \theta } \\
	&= \E \nabla g(Z_0) \cdot H + \E \int_0^1{ \left( \nabla g(Z_0 + \theta H) - \nabla g(Z_0)\right)\cdot H \d \theta } \\
	&= \E \nabla g(Z_0) \cdot H + O_T(\E |H|^2). \end{align*}
	So $u^x_t$ is Lions differentiable at $\nu_\infty$ with $\partial_\nu u^x_t(\nu_\infty)(y)= \nabla g(y)$. In addition:
		\begin{align*} \E \nabla g(Z_0) \cdot H = \E [\nabla g(Z_0) \cdot \E (H ~|~Z_0)] &= \E \nabla g(Z_0) \cdot h(Z_0) = \Theta_t(h)(x).  \end{align*}
	The third point is a direct consequence of Proposition~\ref{prop:control hnu phinu Omeganu} and of the inequality:
	\[ {W_1(\nu, \nu_\infty)}^2 \leq {W_2(\nu, \nu_\infty)}^2 \leq \E |H|^2. \] 
	To check the last point, it suffices to show that
	\[ \E [ \partial_\nu v^x_t(\nu_\infty)(Z_0) \cdot h(Z_0) ]= \Omega_t(h),  \]
where $\partial_\nu v^x_t(\nu_\infty)(y)$ is given by the right-hand side of~\eqref{eq:lions derivative McKean}. This equality follows by the linearity of $h \mapsto \Omega_t(h)$ and by Fubini's theorem.
\end{proof}

\subsection{Static bifurcation analysis: a Green-Kubo formula}\label{sec:bifurcations}
Our result also provides some information on the number of invariant probability measures of~\eqref{eq:McKeanVlasov1}.
In this Section, in addition to Assumption~\ref{ass:standing assumptions1}, we assume that $F$ is bounded
\[ \sup_{x, y \in \R^d} \abs{F(x, y)} < \infty, \]
and that the drift is confining in the sense that:
\[ \exists R, \beta > 0, \quad |x-y| \geq R \implies (b(x) - b(y))\cdot (x-y) \leq -\beta |x-y|^2. \]
We consider the following open ball of $\cK$
\[ \cK_0 := \{ \alpha \in \cK: ~ \norm{\alpha}_{\cK} \leq \norm{F}_\infty +  \norm{\nabla_x F}_\infty \}. \]
Let $\alpha \in \cK_0$. Then Assumption~\ref{ass:standing assumption2} holds with constants $R, \beta$ independent of $\alpha$.
So the SDE~\eqref{eq:linear non-homogeneous SDE} has a unique invariant probability measure, denoted by $\nu^\alpha_\infty$.  In addition, the bound~\eqref{eq:W1 bound eberle} holds, for some  constants $C_*$ and $\kappa_*$ which do not depend on $\alpha$.
We consider $\Psi: \cK_0 \rightarrow \cK_0 $, defined by:
\[ \forall \alpha \in \cK_0, \quad  \Psi(\alpha) := x \mapsto \int_{\R^d} F(x, y) \nu^{\alpha}_\infty(\d y). \]
Note that there is a one-to-one correspondence between the invariant probability measures of~\eqref{eq:McKeanVlasov1} and the fixed-points of $\Psi$ in $\cK_0$.
The following result is a generalization of the Green-Kubo formula~\cite[Ch. 5]{greenkubo}:
\begin{proposition}\label{prop:static bifurcation}
	The function $\Psi$ is Frechet differentiable at every $\alpha \in \cK_0$ and
	\[ D_\alpha  \Psi(\alpha) \cdot \epsilon = \int_0^\infty{ 
\Theta_t(\epsilon) \d t} = \hat{\Theta}(\epsilon) (0), \quad \epsilon, \alpha \in \cK_0, \]
	where $\Theta_t$ is given by~\eqref{eq: definition of Theta t i j}. 
\end{proposition}
\begin{proof}
	We have for all $T \geq 0$, 
	\begin{align*}
		\int_{\R^d} F(x, y) (\nu_\infty^{\alpha+\epsilon} - \nu_\infty^\alpha)(\d y) &= \left[ \E F(x, Y^{{\alpha+\epsilon}, 
			\nu_\infty^{\alpha+\epsilon}}_T) - \E 
		F(x, Y^{\alpha+\epsilon, 
			\nu_\infty^\alpha}_T) \right] \\
											     & \quad + \left[ \E F(x, Y^{\alpha+\epsilon, 
			\nu_\infty^\alpha}_T) -  \E F(x, Y^{\alpha, \nu_\infty^\alpha}_T) \right] \\
		&=: A(x) + B(x).
	\end{align*}
	By~\eqref{eq:W1 bound eberle}, there exist positive constants $C, \kappa_*$ such that 
	$\norm{A}_\cK \leq C e^{-\kappa_* T} W_1 (\nu^{\alpha+\epsilon}_\infty, \nu^\alpha_\infty)$: this term can be made arbitrarily small by choosing $T$ sufficiently large.  In addition, using 
	Proposition~\ref{prop:integrated sensitivity formula}, we have
	\[ B(x) = \int_0^T{ \int_{\mathbb{R}^d}{ \left[ \nabla_y \E_y F(x, Y^{\alpha}_{T-\theta}) \cdot \epsilon(y)  \right]
			\Law(Y^{\alpha+\epsilon, \nu^\alpha_\infty}_\theta)(\d y) \d \theta}  }. \]
			It follows that $\norm{B}_\cK \leq C \norm{\epsilon}_\infty$. Letting $T \rightarrow \infty$ proves that $\Psi$ is continuous. 
By refining the previous argument, we show that this function is Frechet differentiable with the stated derivative. Define $G^x_t(y) :=  \nabla_y \E_y F(x, Y^{\alpha}_t) \cdot \epsilon(y)$. We have $B(x) = \int_0^T{  \E G^x_{T- \theta}(Y^{\alpha+\epsilon, \nu_\infty^\alpha}_\theta) \d \theta}$ 
and, by Girsanov's theorem, provided that $\norm{\epsilon}_\infty^2 T < 1$, we have: 
\begin{align*}  \left| \E G^x_{T- \theta}(Y^{\alpha+\epsilon, \nu_\infty^\alpha}_\theta) -  \E G^x_{T- \theta}(Y^{\alpha, 
	\nu_\infty^\alpha}_\theta) \right| & \leq C \norm{G^x_{T-\theta}}_\infty \sqrt{\theta} \norm{\epsilon}_\infty  \\
	& \leq C e^{-\kappa_*(T-\theta)}  \norm{\epsilon}^2_\infty \sqrt{T}.  
\end{align*}
	Therefore, we deduce that
\[ \left| B(x) - \int_0^\infty{\int_{\mathbb{R}^d}{ \nabla_y 
			\E_y 
F(x, Y^{\alpha}_{\theta}) \cdot \epsilon(y) \nu_\infty^\alpha(\d y)} \d \theta} \right| \leq C( \sqrt{T} \norm{\epsilon}^2_\infty  +  e^{-\kappa_* T}).  \] 
	Overall, there exists a constant $C$ such that 
	\begin{align*}
	& 	\left| \int_{\R^d} F(x, y) (\nu^{\alpha+\epsilon}_\infty - \nu^\alpha_\infty)(\d y) - \int_0^\infty{\int_{\mathbb{R}^d}{ \nabla_y 
			\E_y 
			F(x, Y^{\alpha}_{\theta}) \cdot \epsilon(y) \nu_\infty^\alpha(\d y)} \d \theta} \right| \\
	& \quad \quad  \leq C [e^{-\kappa_* T} + 
	\norm{\epsilon}_\infty^2 \sqrt{T}].  
	\end{align*}
	We choose $T = 1/\norm{\epsilon}_\infty$ and let $\norm{\epsilon}_\infty$ goes to zero: the right-hand term is a $o(\norm{\epsilon}_\infty)$.
	The same estimate holds on the derivative of $\Psi$ with respect to $x$:
	\begin{align*}
	& 	\left| \int_{\R^d} \nabla_x F(x, y) (\nu^{\alpha+\epsilon}_\infty - \nu^\alpha_\infty)(\d y) - \int_0^\infty{\int_{\mathbb{R}^d}{ \nabla_y 
			\E_y 
			\nabla_x F(x, Y^{\alpha}_{\theta}) \cdot \epsilon(y) \nu_\infty^\alpha(\d y)} \d \theta} \right| \\
			& \quad \quad  \leq C [e^{-\kappa_* T} + 
		\norm{\epsilon}_\infty^2 \sqrt{T}].  
	\end{align*}
	Altogether, we deduce that $\alpha 
	\mapsto \int_{\R^d} F(\cdot, y) \nu^\alpha_\infty(\d y) $ is Frechet differentiable with the stated derivative.
\end{proof}
The interpretation of this result is the following: static bifurcations, leading to a change of the number of invariant probability measures of~\eqref{eq:McKeanVlasov1}, occur for parameters satisfying $\det(I - \hat{\Theta}(0)) = 0$. On the other hand, we expect Hopf bifurcations to occur at parameters for which $\det(I - \hat{\Theta}(i \omega)) = 0$, for some $\omega > 0$. Altogether, this covers the two canonical ways to break the stability condition $\lambda' > 0$, where $\lambda'$ is given by~\eqref{eq:spectral assumption}. The study of these bifurcations is left to future research. 
\section{McKean-Vlasov of convolution type on the torus}\label{sec:part2}
Let $\beta > 0$. We consider the following McKean-Vlasov equation on the torus $\T^d := {(\mathbb{R} /  2 \pi \mathbb{Z})}^d$:
\begin{equation}
	\label{eq:McKeanVlasov2} 
	\d X^\nu_t = - \int_{\mathbb{T}}{ \nabla W (X^\nu_t-y) \mu_t(\d y)} \d t + \sqrt{2 \beta^{-1}} \d B_t \quad \text{ with } \quad \mu_t = \Law(X^\nu_t).
\end{equation}
with initial condition $\Law(X^\nu_0) = \nu \in \cP(\T^d)$.
Here $(B_t)$ is a Brownian motion on $\T^d$. 
This equation generalizes the Kuramoto model~\cite{RevModPhys.77.137, MR2594897, MR3689966, zbMATH06040434}, for 
which $d = 1$ and $W = -\kappa \cos$ for some constant $\kappa \geq 0$. 
We refer to~\cite{MR4062483} for a detailed presentation of examples that fit in the framework of~\eqref{eq:McKeanVlasov2}, as well as a study of the static bifurcations of this equation.
In this Section, we study the local stability of the uniform probability measure using the strategy and the tools introduced in Section~\ref{sec:part1}. 
\subsection{Main result}
Write the interaction kernel $W: \T^d \rightarrow \R$ in Fourier:
\begin{equation}
	\label{eq:fourier series W}
	 W(x) = \sum_{n \in \mathbb{Z}^d} \tilde{W}(n) e^{i n \cdot x}, \quad x \in \mathbb{T}^d,
\end{equation}
where $n \cdot x = \sum_{i = 1}^d n_i x_i$. Let $|n|^2 = n \cdot n$. The Fourier coefficients of $W$ are given by
\[ \tilde{W}(n) = \frac{1}{{(2 \pi)}^d} \int_{\T^d} W(y) e^{-i n \cdot y} \d y, \quad n \in \mathbb{Z}^d. \]
\begin{assumption}\label{hyp:WC3}
Assume that $W \in C^3(\T^d)$ and that $\sum_{n \in \mathbb{Z}^d} { |n|^2 |\tilde{W}(n)| } < \infty$.\end{assumption}
The uniform probability measure
\[ U(\d x) := \frac{\d x}{{(2\pi)}^d} \]
is an invariant probability measure of~\eqref{eq:McKeanVlasov2} and
\begin{theorem}\label{th:second main result}
	In addition to Assumption~\ref{hyp:WC3}, assume that
	\begin{equation} 
		\label{eq:spectral assumption torus}
		\lambda' := \inf_{n \in \mathbb{Z}^d \setminus \{0\} } |n|^2 \left( \beta^{-1} +  \Re(\tilde{W}(n)) \right) > 0. 
	\end{equation}
	Then $U(\d x) = \frac{\d x}{{(2 \pi)}^d}$ is locally stable: there exists $\lambda \in (0, \lambda')$, $\epsilon > 0$ and $C > 1$ such that for all $\nu \in \cP(\T^d)$ with $W_1(\nu, U) < \epsilon$, it holds that
	 \[ \forall t \geq 0, \quad  W_1(\Law(X^\nu_t), U) \leq C W_1(\nu, U) e^{-\lambda t}. \]
\end{theorem}
\begin{remark}
	The constant $\lambda'$ is the exact analogue of \eqref{eq:spectral assumption} in Section~\ref{sec:part1}.
	When the interaction kernel $W$ is even,~\cite{MR4062483} studies the existence of bifurcations of the invariant probability measures of~\eqref{eq:McKeanVlasov2}, provided that there exists $n \in \mathbb{Z}^d \setminus \{0\}$ such that $\beta^{-1} + \Re(\tilde{W}(n)) = 0$. Therefore, criterion~\eqref{eq:spectral assumption torus} is sharp: we prove that the uniform measure is stable up to the first bifurcation. Note that we do not require here $W$ to be even. 
\end{remark}
\subsection{Proof}
To simplify the notations, we first assume that $d = 1$. 
We discuss the case $d > 1$ afterward,  most of the arguments being  the same. We write $\sigma := \sqrt{2 \beta^{-1}}$. The proof is divided into the following steps.\begin{enumerate}[label=\textbf{Step \arabic*.}, wide, labelwidth=1pt, labelindent=0pt] 
	\item Because $\nabla W$ is Lipschitz, the equation~\eqref{eq:McKeanVlasov2} has a unique path-wise solution satisfying the following apriori estimate:
		\[  \forall T > 0, \exists C_T: \forall \nu, \mu \in \cP(\T), \quad \sup_{t \in [0, T]}W_1(\Law(X^\nu_t), \Law(X^\mu_t)) \leq C_T W_1(\nu, \mu). \]
\item We define for $\nu \in \cP(\T)$, $x \in \T$ and $t \geq 0$: 
	\[ k^\nu_t(x) := -\E \nabla W(x- X^\nu_t). \]
	Recall that $U(\d x) = \frac{\d x}{2 \pi}$. Because $k^U_t \equiv 0$, we have, by Step 1
	\[
		\norm{k^\nu_t}_\infty = \sup_{x \in \T}|k^\nu_t(x) - k^{U}_t(x)| \leq C_T \norm{\nabla^2 W}_\infty W_1(\nu, U). 
	\]
	In addition, $x \mapsto k^\nu_t(x)$ is differentiable and
	\[
		\norm{\nabla k^\nu_t}_\infty = \sup_{x \in \T} |\nabla k^\nu_t(x) - \nabla k^U_t(x)| \leq C_T \norm{\nabla^3 W}_\infty W_1(\nu, U). 
	\]
\item We now use that there exists $C_* > 1$ and $\kappa_* > 0$ such that 
	\[ \forall x,y \in \T, \forall t \geq 0,  \quad W_1(\Law(x+\sigma B_t), \Law(y + \sigma B_t)) \leq C_* e^{- \kappa_* t} |x-y|.\]
We refer to~\cite[Prop. 4]{MonmarcheJournel}. We define for all $t \geq 0$, $x \in \T$ and $\nu \in \cP(\T)$:
\[ \varphi^\nu_t(x) := -\E \nabla W(x - X^\nu_0 -\sigma B_t), \] where $X^\nu_0$ is independent of ${(B_t)}_{t \geq 0}$ and has law $\nu$. Because $\norm{\nabla^2 W}_\infty < \infty$, by the preceding result and the dual formulation of the $W_1$ norm, there exists a constant $C > 0$ such that:
\[ \norm{\varphi^\nu_t}_\infty \leq C e^{-\kappa_* t } W_1(\nu, U). \]
\item Let $\cK := W^{1,\infty}(\T)$ be the space of bounded and Lipschitz continuous functions from $\T$ to $\R$. For $k \in C(\R_+; \cK)$ and $\nu \in \cP(\T)$, we consider $(Y^{k, \nu}_t)$ the solution of the following linear non-homogeneous SDE
	\[ \d Y^{k, \nu}_{t} = k_t(Y^{k, \nu}_t) \d t + \sigma \d B_t, \]
	starting with $\Law(Y^{k, \nu}_0) = \nu$.
	Let $g \in C^2(\T)$. The integrated sensibility formula of Proposition~\ref{prop:integrated sensitivity formula} is, in this context:
	\[ \E g(Y^{k, \nu}_t) - \E g(Y^{0, \nu}_t) = \int_0^t \int_\T \nabla_y \E_y g(y + \sigma B_{t - \theta}) \cdot k_\theta(y) \Law(Y^{k, \nu}_\theta)(\d y) \d \theta. \]
\item We let for $h \in L^2(\T)$ and $x \in \T$:
	\[  \Theta_t(h)(x) :=  - \int_\T \nabla_y \E_y \nabla W (x - y - \sigma B_t) \cdot h(y) \frac{\d y}{2 \pi}.  \]
	Using that $\E e^{in \sigma B_t} = e^{-\frac{n^2 \sigma^2}{2}t} = e^{-\frac{n^2 t}{\beta}}$, we find that the Fourier series of $\Theta_t(h)(x)$ is:
	\[  \Theta_t(h)(x) = - \sum_{n \in \mathbb{Z}} n^2 \tilde{W}(n) \tilde{h}(n) e^{-\frac{n^2 t}{\beta}} e^{inx}.   \]
	So $\Theta_t$ is diagonal in the Fourier basis ${(e^{inx})}_{n \in \mathbb{Z}}$ and $\widetilde{\Theta_t(h)}(n) = -n^2 \tilde{W}(n) e^{-\frac{n^2 t}{\beta}} \tilde{h}(n)$.
	In addition, using $\abs{\tilde{h}(n)} \leq \norm{h}_\infty$ we have:
	\[ \norm{\Theta_t(h)}_\infty \leq C_0 e^{- t / \beta} \norm{h}_\infty,     \]
	where $C_0 := \sum_{n \in \mathbb{Z}} n^2 \left|\tilde{W}(n) \right| < \infty$.
\item We then define $\Omega_t(h)$ to be the unique solution of the Volterra integral equation: 
	\[ \forall t \geq 0, \quad \Omega_t(h)= \Theta_t(h) + \int_0^t{ \Theta_{t-s}(\Omega_s(h)) \d s }. \]
	Again, $\Omega_t$ is diagonal in the Fourier basis:
	\[ \Omega_t(h)(x)  = -\sum_{n \in \mathbb{Z}} n^2 \tilde{W}(n) \exp \left(-n^2t \left[\beta^{-1}  + \tilde{W}(n) \right] \right) \tilde{h}(n) e^{inx}. \]
	Let $\lambda'$ be given by~\eqref{eq:spectral assumption torus}. We have:
	\[  \norm{\Omega_t(h)}_{\infty} \leq C_0 e^{-\lambda' t} \norm{h}_{\infty}. \]
	So, under the condition $\lambda' > 0$, $(\Omega_t)$ decays at an exponential rate towards zero.
\item Let $x \in \T$ be fixed. We now apply Step 4 with $g(y) := -\nabla W(x-y)$,  and with $k_t(y) := k^\nu_t(y)$, where $k^\nu_t$ is defined in Step 2. Note that with this choice, $Y^{k, \nu}_t = X^\nu_t$ and so $\E g(Y^{k, \nu}_t) = k^\nu_t(x)$. Similarly, $\E g(Y^{0, \nu}_t) = \varphi^\nu_t(x)$, where $\varphi^\nu_t(x)$ is defined in Step 3. 
	Therefore, we have:
	\begin{align*} k^\nu_t(x) - \varphi^\nu_t(x) &= \int_0^t \int_{\T} \E \nabla^2 W(x-y-\sigma B_{t-\theta}) \cdot k^\nu_\theta(y) \Law(X^\nu_\theta)(\d y) \d \theta \\
		&= \int_0^t{\Theta_{t-\theta}(k^\nu_\theta)(x) \d \theta} + R_t(x),
	\end{align*}
	where
	\begin{align*} R_t(x) &:= \int_0^t \E \left[ G^x_{t,\theta}(X^\nu_\theta) - G^x_{t, \theta}(X^U_\theta) \right] \d \theta, \\
	G^x_{t, \theta}(y) &:= \E \nabla^2 W (x-y -\sigma B_{t-\theta}) \cdot k^\nu_\theta(y) .
	\end{align*}
	Using the apriori estimates of Step 2, we deduce that there exists a constant $C_T$ such that for all $0 \leq \theta \leq t \leq T$:
	\[ |\nabla_y G^x_{t, \theta}(y)| \leq C_T W_1(\nu, U). \]
	Using Step 1, we conclude that $\abs{R_t(x)} \leq C_T {(W_1(\nu, U))}^2$.
	To summarize, we have proven that for all $T > 0$, there exists a constant $C_T$ such that for all $\nu \in \cP(\T)$ and for all $t \in [0, T]$:
	\[ \left| k^\nu_t(x) -  \varphi^\nu_t(x) - \int_0^t{ \Theta_{t-\theta}(k^\nu_s)(x) \d \theta}  \right| \leq  C_T {(W_1(\nu, U))}^2. \]
\item By iterating the last inequality of Step 7, we obtain that for all $T > 0$, there exists a constant $C_T$ such that
	\[ \left| k^\nu_t(x) -  \varphi^\nu_t(x) - \int_0^t{ \Omega_{t-\theta}(\varphi^\nu_s)(x) \d \theta}  \right| \leq  C_T {(W_1(\nu, U))}^2. \]
\item We prove that there exists a constant $C > 0$ such that for all $t > 0$, for all $k \in C([0, t]; \cK)$ and for all $\nu \in \cP(\T)$, it holds that
	\[ W_1(\Law(Y^{k,\nu}_t), \Law(Y^{0, \nu}_t)) \leq C \int_0^t e^{-\kappa_*(t-\theta)} \norm{k_\theta}_\infty \d \theta. \]
	The proof is obtained exactly as in Corollary~\ref{cor:wassertein control with h}; it uses the estimates of Step 3.
\item We fix $\lambda \in (0, \min(\kappa_*, \lambda'))$.
	Using Step 8, Step 6, and Step 3, we deduce that there exists $C_\lambda > 0$ such that for all $T > 0$, there is a constant $C_{T}$ such that for all $t \in [0, T]$ and $\nu \in \cP(\T)$:
	\[ \norm{k^\nu_t}_\infty \leq C_{T} {\left( W_1(\nu, U)\right)}^2 + C_\lambda W_1(\nu, U) e^{-\lambda t}. \]
	Let $k_t(x) := k^\nu_t(x)$. Using that $X^\nu_t = Y^{k, \nu}_t$, we have:
	\[ W_1(\Law(X^\nu_t), U) \leq W_1(\Law(Y^{k, \nu}_t), \Law(Y^{0, \nu}_t)) + W_1(\Law(Y^{0, \nu}_t), U). \]
	By Step 3, we have
	\[ W_1(\Law(Y^{0, \nu}_t), U) \leq C_* e^{-\kappa_* t} W_1(\nu, U). \]
	By Step 9, we have
	\[  W_1(\Law(Y^{k, \nu}_t), \Law(Y^{0, \nu}_t)) \leq C \int_0^t{ e^{-\kappa_*(t-\theta)} \norm{k^\nu_\theta}_\infty \d \theta}.  \]
	Altogether, we deduce that there is a constant $C_\lambda$ such that for all $T > 0$, there exists $C_T > 0$ such that for all $t \in [0, T]$, for all $\nu \in \cP(\T)$, we have:
	\[ W_1(\Law(X^\nu_t), U) \leq C_\lambda W_1(\nu, U) e^{-\lambda t} + C_T {\left( W_1(\nu, U) \right)}^2. \]
	The proof of Theorem~\ref{th:second main result} is deduced from this estimate, exactly as we did at the end of Section~\ref{sec:proof of main result 1}. This ends the proof for $d = 1$.
	\end{enumerate}
	The case $d > 1$ is similar; the only differences are in the expressions of  $\Theta_t$ and $\Omega_t$ of Steps~5 and 6. Given $n \in \mathbb{Z}^d$, we denote by $P_{(n)}$ the $d \times d$ matrix defined by $P_{(n)} = {(n_i n_j)}_{1 \leq i,j \leq d}$.
	We find that for all $h \in L^2(\T^d; \R^\d)$ and for all $x \in \T^d$,  
	\[ \Theta_t(h)(x) = - \sum_{n \in \mathbb{Z}^d} e^{i n \cdot x }\tilde{W}(n) e^{-\frac{|n|^2 t}{\beta}} P_{(n)} \tilde{h}(n),\] 
	and
	\[  \Omega_t(h)(x) = -\sum_{n \in \mathbb{Z}^d} e^{i n \cdot x } \tilde{W}(n) e^{-\frac{|n|^2 t}{\beta} }  P_{(n)} e^{- t \tilde{W}(n) P_{(n)}} \tilde{h}(n). \]
	The eigenvalues of $P_{(n)}$ are $|n|^2$ (of order 1) and zero (of order $d-1$). In addition, it holds that for $\theta \in \R$,
	\[ {(e^{\theta P_{(n)}})}_{i,j} = \delta_{\{i = j\}} + \frac{n_i n_j}{|n|^2}  (e^{\theta |n|^2} -1 ). \]
	Therefore, the estimates of Steps 5 and 6 still hold in dimension $d > 1$. This ends the proof.
\bibliographystyle{abbrvnat}
\appendix
\section{Appendix}
\subsection*{Proof of Lemma~\ref{lem:sigma is equal to zero}}
	Recall that $\cV(x,y) := b(x) + F(x,y)$.
	When $\sigma \equiv 0$, then $\nu_\infty = \delta_{x_*}$ for some $x_* \in \R^d$. Therefore~\eqref{eq: definition of Theta t i j} writes:
	\[ \Theta_t(h)(x) = \nabla_y \cV(x, x_*) e^{t \nabla_x \cV (x_*, x_*)} \cdot h(x_*) =: A^x_t \cdot h(x_*).  \]
	We look at solutions of~\eqref{eq:volterra equation Theta Omega} of the form:
	$\Omega_t(h)(x) := B^x_t \cdot h(x_*)$, for some matrices $B^x_t$. We find that $B^{x}_t$ satisfies
	\begin{equation}
		\label{eq:volterra integral equation matrices A and B}
		\forall t \geq 0, \quad B^x_t = A^x_t + \int_0^t{ A^x_{t-s} \cdot B^{x_*}_{s} \d s}. 
	\end{equation}
	We first study this equation for $x = x_*$. 
	For all $z \in \mathbb{C}$ with $\Re(z)
 > -\kappa$,
 \[ \widehat{A^{x_*}}(z) = \nabla_y \cV(x_*, x_*)  \int_0^\infty{
 e^{-t(z I_d - \nabla_x \cV (x_*, x_*))} \d t} = \nabla_y \cV(x_*, x_*)  {(z I_d - \nabla_x \cV(x_*, x_*))}^{-1}.  \]
 \begin{lemma}\label{lem:A1.jacobian}
	 Let $z \in \C$ such that $\Re(z) > -\kappa$. Then $\det \left(I_d - \widehat{A^{x_*}}(z) \right) = 0$ if and only if $z$ is an eigenvalue of $\nabla_x \cV(x_*, x_*) + \nabla_y \cV(x_*, x_*)$.
 \end{lemma}
 \begin{proof}
When $\det \left(I_d - \widehat{A^{x_*}}(z) \right) = 0$, there exists $u \in \mathbb{R}^d \setminus  \{0\}$ such
        that 
\[ u  = \nabla_y \cV(x_*, x_*)  {(z I_d - \nabla_x \cV (x_*, x_*))}^{-1} u. \]
Setting $v = {(z I_d - \nabla_x \cV(x_*, x_*))}^{-1} u$, we have $v \neq 0$ and
        \[ z v = (\nabla_x \cV(x_*, x_*) + \nabla_y \cV(x_*, x_*)) v.  \]
        So $z$ is an eigenvalue of $\nabla_x \cV(x_*, x_*) + \nabla_y \cV(x_*, x_*)$. The converse statement is proved similarly.
\end{proof}
	Now, assume that all the real parts of the eigenvalues of $\nabla_x \cV(x_*, x_*) + \nabla_y \cV(x_*, x_*)$ are less than $- \lambda '$, for some $\lambda' \in (0, \kappa)$. Then, for all $\Re(z) \geq  -\lambda'$, $\det(I - A^{x_*}_t) \neq 0$. Applying~\cite[Th. 4.1]{MR1050319}, we deduce that $\int_0^\infty{e^{\lambda' t} \norm{B^{x_*}_t} \d t} < \infty$. Using~\eqref{eq:volterra integral equation matrices A and B}, one deduces that $\norm{B^{x_*}_t} e^{\lambda' t} < \infty$, and so
	\[  \norm{\Omega_t(h)}_{\cH} = |\Omega_t(h)(x_*)| = |B^{x_*}_t \cdot h(x_*)| \leq Ce^{-\lambda' t} |h(x_*)| = Ce^{-\lambda't} \norm{h}_{\cH}.  \]
	Therefore, $\lambda' > 0$. Conversely, if $\lambda' > 0$ where $\lambda'$ is given by~\eqref{eq:spectral assumption} holds, then for $\Re(z) > -\lambda'$, $\det(I_d - \widehat{A^{x_*}}(z)) \neq 0$.  So by Lemma~\ref{lem:A1.jacobian}, $z$ is not an eigenvalue of $\nabla_x \cV(x_*, x_*) + \nabla_y \cV(x_*, x_*)$. We deduce that all the eigenvalues of this matrix have real parts less or equal to $-\lambda'$.
	\subsection*{Proof of Proposition~\ref{prop:sigma-small-stability}}
	Denote by $\bar{\Theta}_t := \Theta^\sigma_t - \Theta^0_t$ and let $L_t(h) := \bar{\Theta}_t(h) + \int_0^t \Omega^0_{t-s}( \bar{\Theta}_s(h)) \d s$. We also let $R$ be the resolvent of $L$, such that $R$ solves
	\[ R_t = L_t + \int_0^t L_{t-s} \cdot R_s \d s. \]
\begin{lemma}\label{lem:dvlp-Omega-sigma}
	It holds that for all $t \geq 0$, $\Omega^\sigma_t = \Omega^0_t + R_t + \int_0^t R_{t-s} \cdot \Omega^0_s \d s$.
\end{lemma}
\begin{proof}
Let $Q_t := \Omega^0_t + R_t + \int_0^t R_{t-s} \cdot \Omega^0_s \d s$. To simplify the notation, we write for $A, B \in \cL(\cH)$:
\[ {(A*B)}_t= \int_0^t A_{t-s} \cdot B_s \d s.    \]
We have $L*Q = L*\Omega^0 + L * R + (L*R*\Omega^0) = R * \Omega^0 + L*R$.
Therefore, using that $L*R = R-L$ and the definition of $Q$, we find that $Q$ solves
\begin{equation} \label{eq:volterra-eq-Q} Q = \Omega^0 + L + L * Q. \end{equation}
As $L = \bar{\Theta} + \Omega^0 * \bar{\Theta}$, we have 
\[ Q - (\bar{\Theta} + \Omega^0 * \bar{\Theta})*Q =  \Omega^0 +   \bar{\Theta} + \Omega^0 * \bar{\Theta}. \]
We multiply on the left by $\Theta^0$. Using that $\Theta^0 * \Omega^0 = \Omega^0 - \Theta^0$, we find that
\[ \Theta^0 * Q - \Omega^0 *  \bar{\Theta} * Q = \Omega^0 - \Theta^0 + \Omega^0 * \bar{\Theta}.   \]
Finally, using that $\bar{\Theta} = \Theta^\sigma - \Theta^0$, we find that $\Theta^0 - \Omega^0 * \bar{\Theta} * Q = \Omega^0 * \Theta^\sigma$. Altogether:
\[ L*Q = \Theta^\sigma * Q - \Omega^0 * \Theta^\sigma. \]
We substitute this equality in~\eqref{eq:volterra-eq-Q} to finally obtain that
\[ Q = \Theta^\sigma + \Theta^\sigma * Q. \]
As $\Omega^\sigma$ is the unique solution of this Volterra integral equation, we deduce that $\Omega^\sigma = Q$ as claimed.
\end{proof}
By Assumptions, there exists $\sigma_0, \kappa > 0$ small enough such that for all $\sigma \in M_d(\R)$ with $\det(\sigma) > 0$ and $\norm{\sigma}_2 \leq \sigma_0$, we have
\[ \sup_{t \geq 0}e^{\kappa t} \norm{L_t}_{\cL(\cH)} \leq \frac{\kappa}{4}. \]
Therefore, we deduce that $\sup_{t \geq 0} e^{\frac{\kappa}{2} t} \norm{R_t}_{\cL(\cH)} \leq \frac{\kappa}{2}$ and so the stated result follows using Lemma~\ref{lem:dvlp-Omega-sigma}.
\bibliography{biblio}

\begin{thebibliography}{51}
\providecommand{\natexlab}[1]{#1}
\providecommand{\url}[1]{\texttt{#1}}
\expandafter\ifx\csname urlstyle\endcsname\relax
  \providecommand{\doi}[1]{doi: #1}\else
  \providecommand{\doi}{doi: \begingroup \urlstyle{rm}\Url}\fi

\bibitem[Acebr\'on et~al.(2005)Acebr\'on, Bonilla, P\'erez~Vicente, Ritort, and Spigler]{RevModPhys.77.137}
J.~A. Acebr\'on, L.~L. Bonilla, C.~J. P\'erez~Vicente, F.~Ritort, and R.~Spigler.
\newblock The kuramoto model: A simple paradigm for synchronization phenomena.
\newblock \emph{Rev. Mod. Phys.}, 77:\penalty0 137--185, 2005.
\newblock \doi{10.1103/RevModPhys.77.137}.
\newblock URL \url{https://link.aps.org/doi/10.1103/RevModPhys.77.137}.

\bibitem[Arnaudon and Del~Moral(2020)]{MR4187123}
M.~Arnaudon and P.~Del~Moral.
\newblock A second order analysis of {M}c{K}ean-{V}lasov semigroups.
\newblock \emph{Ann. Appl. Probab.}, 30\penalty0 (6):\penalty0 2613--2664, 2020.
\newblock ISSN 1050-5164.
\newblock \doi{10.1214/20-AAP1568}.
\newblock URL \url{https://doi.org/10.1214/20-AAP1568}.

\bibitem[Benachour et~al.(1998{\natexlab{a}})Benachour, Roynette, Talay, and Vallois]{MR1632193}
S.~Benachour, B.~Roynette, D.~Talay, and P.~Vallois.
\newblock Nonlinear self-stabilizing processes. {I}. {E}xistence, invariant probability, propagation of chaos.
\newblock \emph{Stochastic Process. Appl.}, 75\penalty0 (2):\penalty0 173--201, 1998{\natexlab{a}}.
\newblock ISSN 0304-4149.
\newblock \doi{10.1016/S0304-4149(98)00018-0}.
\newblock URL \url{https://doi.org/10.1016/S0304-4149(98)00018-0}.

\bibitem[Benachour et~al.(1998{\natexlab{b}})Benachour, Roynette, and Vallois]{MR1632197}
S.~Benachour, B.~Roynette, and P.~Vallois.
\newblock Nonlinear self-stabilizing processes. {II}. {C}onvergence to invariant probability.
\newblock \emph{Stochastic Process. Appl.}, 75\penalty0 (2):\penalty0 203--224, 1998{\natexlab{b}}.
\newblock ISSN 0304-4149.
\newblock \doi{10.1016/S0304-4149(98)00019-2}.
\newblock URL \url{https://doi.org/10.1016/S0304-4149(98)00019-2}.

\bibitem[Benedetto et~al.(1998)Benedetto, Caglioti, Carrillo, and Pulvirenti]{MR1637274}
D.~Benedetto, E.~Caglioti, J.~A. Carrillo, and M.~Pulvirenti.
\newblock A non-{M}axwellian steady distribution for one-dimensional granular media.
\newblock \emph{J. Statist. Phys.}, 91\penalty0 (5-6):\penalty0 979--990, 1998.
\newblock ISSN 0022-4715.
\newblock \doi{10.1023/A:1023032000560}.
\newblock URL \url{https://doi.org/10.1023/A:1023032000560}.

\bibitem[Bertini et~al.(2010)Bertini, Giacomin, and Pakdaman]{MR2594897}
L.~Bertini, G.~Giacomin, and K.~Pakdaman.
\newblock Dynamical aspects of mean field plane rotators and the {K}uramoto model.
\newblock \emph{J. Stat. Phys.}, 138\penalty0 (1-3):\penalty0 270--290, 2010.
\newblock ISSN 0022-4715.
\newblock \doi{10.1007/s10955-009-9908-9}.
\newblock URL \url{https://doi.org/10.1007/s10955-009-9908-9}.

\bibitem[Bertini et~al.(2014)Bertini, Giacomin, and Poquet]{zbMATH06380861}
L.~Bertini, G.~Giacomin, and C.~Poquet.
\newblock Synchronization and random long time dynamics for mean-field plane rotators.
\newblock \emph{Probab. Theory Relat. Fields}, 160\penalty0 (3-4):\penalty0 593--653, 2014.
\newblock ISSN 0178-8051.
\newblock \doi{10.1007/s00440-013-0536-6}.

\bibitem[Bolley et~al.(2011)Bolley, Ca\~{n}izo, and Carrillo]{MR2860672}
F.~Bolley, J.~A. Ca\~{n}izo, and J.~A. Carrillo.
\newblock Stochastic mean-field limit: non-{L}ipschitz forces and swarming.
\newblock \emph{Math. Models Methods Appl. Sci.}, 21\penalty0 (11):\penalty0 2179--2210, 2011.
\newblock ISSN 0218-2025.
\newblock \doi{10.1142/S0218202511005702}.
\newblock URL \url{https://doi.org/10.1142/S0218202511005702}.

\bibitem[Butkovsky(2014)]{MR3403022}
O.~A. Butkovsky.
\newblock On ergodic properties of nonlinear {M}arkov chains and stochastic {M}c{K}ean-{V}lasov equations.
\newblock \emph{Theory Probab. Appl.}, 58\penalty0 (4):\penalty0 661--674, 2014.
\newblock ISSN 0040-585X.
\newblock \doi{10.1137/S0040585X97986825}.
\newblock URL \url{https://doi.org/10.1137/S0040585X97986825}.

\bibitem[Carmona and Delarue(2018)]{MR3752669}
R.~Carmona and F.~Delarue.
\newblock \emph{Probabilistic theory of mean field games with applications. {I}}, volume~83 of \emph{Probability Theory and Stochastic Modelling}.
\newblock Springer, Cham, 2018.
\newblock Mean field FBSDEs, control, and games.

\bibitem[Carrillo et~al.(2003)Carrillo, McCann, and Villani]{MR2053570}
J.~A. Carrillo, R.~J. McCann, and C.~Villani.
\newblock Kinetic equilibration rates for granular media and related equations: entropy dissipation and mass transportation estimates.
\newblock \emph{Rev. Mat. Iberoamericana}, 19\penalty0 (3):\penalty0 971--1018, 2003.
\newblock ISSN 0213-2230.
\newblock \doi{10.4171/RMI/376}.
\newblock URL \url{https://doi.org/10.4171/RMI/376}.

\bibitem[Carrillo et~al.(2020)Carrillo, Gvalani, Pavliotis, and Schlichting]{MR4062483}
J.~A. Carrillo, R.~S. Gvalani, G.~A. Pavliotis, and A.~Schlichting.
\newblock Long-time behaviour and phase transitions for the {M}c{K}ean-{V}lasov equation on the torus.
\newblock \emph{Arch. Ration. Mech. Anal.}, 235\penalty0 (1):\penalty0 635--690, 2020.
\newblock ISSN 0003-9527.
\newblock \doi{10.1007/s00205-019-01430-4}.
\newblock URL \url{https://doi.org/10.1007/s00205-019-01430-4}.

\bibitem[Cattiaux and Guillin(2014)]{zbMATH06892760}
P.~Cattiaux and A.~Guillin.
\newblock Semi log-concave {Markov} diffusions.
\newblock In \emph{S\'eminaire de Probabilit\'es XLVI}, pages 231--292. Cham: Springer, 2014.
\newblock ISBN 978-3-319-11969-4; 978-3-319-11970-0.
\newblock \doi{10.1007/978-3-319-11970-0_9}.

\bibitem[Coppini(2022)]{zbMATH07493825}
F.~Coppini.
\newblock Long time dynamics for interacting oscillators on graphs.
\newblock \emph{Ann. Appl. Probab.}, 32\penalty0 (1):\penalty0 360--391, 2022.
\newblock ISSN 1050-5164.
\newblock \doi{10.1214/21-AAP1680}.

\bibitem[Cormier(2021)]{cormier2021meanfield}
Q.~Cormier.
\newblock A mean-field model of integrate-and-fire neurons: non-linear stability of the stationary solutions.
\newblock Technical report, Inria, 2021.

\bibitem[Cormier et~al.(2020)Cormier, Tanr\'{e}, and Veltz]{MR4080722}
Q.~Cormier, E.~Tanr\'{e}, and R.~Veltz.
\newblock Long time behavior of a mean-field model of interacting neurons.
\newblock \emph{Stochastic Process. Appl.}, 130\penalty0 (5):\penalty0 2553--2595, 2020.
\newblock ISSN 0304-4149.
\newblock \doi{10.1016/j.spa.2019.07.010}.
\newblock URL \url{https://doi.org/10.1016/j.spa.2019.07.010}.

\bibitem[Cormier et~al.(2021)Cormier, Tanr\'{e}, and Veltz]{MR4316639}
Q.~Cormier, E.~Tanr\'{e}, and R.~Veltz.
\newblock Hopf bifurcation in a mean-field model of spiking neurons.
\newblock \emph{Electron. J. Probab.}, 26:\penalty0 Paper No. 121, 40, 2021.
\newblock \doi{10.1214/21-ejp688}.
\newblock URL \url{https://doi.org/10.1214/21-ejp688}.

\bibitem[Degond et~al.(2015)Degond, Frouvelle, and Liu]{MR3305654}
P.~Degond, A.~Frouvelle, and J.-G. Liu.
\newblock Phase transitions, hysteresis, and hyperbolicity for self-organized alignment dynamics.
\newblock \emph{Arch. Ration. Mech. Anal.}, 216\penalty0 (1):\penalty0 63--115, 2015.
\newblock ISSN 0003-9527.
\newblock \doi{10.1007/s00205-014-0800-7}.
\newblock URL \url{https://doi.org/10.1007/s00205-014-0800-7}.

\bibitem[Del~Moral and Tugaut(2019)]{MR4020054}
P.~Del~Moral and J.~Tugaut.
\newblock Uniform propagation of chaos and creation of chaos for a class of nonlinear diffusions.
\newblock \emph{Stoch. Anal. Appl.}, 37\penalty0 (6):\penalty0 909--935, 2019.
\newblock ISSN 0736-2994.
\newblock \doi{10.1080/07362994.2019.1622426}.
\newblock URL \url{https://doi.org/10.1080/07362994.2019.1622426}.

\bibitem[Delarue and Tse(2021)]{delarue2021uniform}
F.~Delarue and A.~Tse.
\newblock Uniform in time weak propagation of chaos on the torus, 2021.

\bibitem[Delarue et~al.(2015)Delarue, Inglis, Rubenthaler, and Tanr\'{e}]{MR3349003}
F.~Delarue, J.~Inglis, S.~Rubenthaler, and E.~Tanr\'{e}.
\newblock Global solvability of a networked integrate-and-fire model of {M}c{K}ean-{V}lasov type.
\newblock \emph{Ann. Appl. Probab.}, 25\penalty0 (4):\penalty0 2096--2133, 2015.
\newblock ISSN 1050-5164.
\newblock \doi{10.1214/14-AAP1044}.
\newblock URL \url{https://doi.org/10.1214/14-AAP1044}.

\bibitem[Dobru\v{s}in(1979)]{MR541637}
R.~L. Dobru\v{s}in.
\newblock Vlasov equations.
\newblock \emph{Funktsional. Anal. i Prilozhen.}, 13\penalty0 (2):\penalty0 48--58, 96, 1979.
\newblock ISSN 0374-1990.

\bibitem[Eberle(2016)]{Eberlecontraction}
A.~Eberle.
\newblock Reflection couplings and contraction rates for diffusions.
\newblock \emph{Probab. Theory Relat. Fields}, 166\penalty0 (3-4):\penalty0 851--886, 2016.
\newblock ISSN 0178-8051.
\newblock \doi{10.1007/s00440-015-0673-1}.

\bibitem[Eberle et~al.(2019)Eberle, Guillin, and Zimmer]{MR3939573}
A.~Eberle, A.~Guillin, and R.~Zimmer.
\newblock Quantitative {H}arris-type theorems for diffusions and {M}c{K}ean-{V}lasov processes.
\newblock \emph{Trans. Amer. Math. Soc.}, 371\penalty0 (10):\penalty0 7135--7173, 2019.
\newblock ISSN 0002-9947.
\newblock \doi{10.1090/tran/7576}.
\newblock URL \url{https://doi.org/10.1090/tran/7576}.

\bibitem[Fournier and L\"{o}cherbach(2016)]{MR3573298}
N.~Fournier and E.~L\"{o}cherbach.
\newblock On a toy model of interacting neurons.
\newblock \emph{Ann. Inst. Henri Poincar\'{e} Probab. Stat.}, 52\penalty0 (4):\penalty0 1844--1876, 2016.
\newblock ISSN 0246-0203.
\newblock \doi{10.1214/15-AIHP701}.
\newblock URL \url{https://doi.org/10.1214/15-AIHP701}.

\bibitem[Fournier et~al.(2014)Fournier, Hauray, and Mischler]{MR3254330}
N.~Fournier, M.~Hauray, and S.~Mischler.
\newblock Propagation of chaos for the 2{D} viscous vortex model.
\newblock \emph{J. Eur. Math. Soc. (JEMS)}, 16\penalty0 (7):\penalty0 1423--1466, 2014.
\newblock ISSN 1435-9855.
\newblock \doi{10.4171/JEMS/465}.
\newblock URL \url{https://doi.org/10.4171/JEMS/465}.

\bibitem[Ganz~Bustos(2008)]{ganzb2008}
A.~Ganz~Bustos.
\newblock \emph{Approximations des distributions d'équilibre de certains systèmes stochastiques avec interactions McKean-Vlasov}.
\newblock PhD thesis, Université de Nice, 2008.
\newblock URL \url{http://www.theses.fr/2008NICE4089}.

\bibitem[Giacomin et~al.(2012)Giacomin, Pakdaman, and Pellegrin]{zbMATH06040434}
G.~Giacomin, K.~Pakdaman, and X.~Pellegrin.
\newblock Global attractor and asymptotic dynamics in the {Kuramoto} model for coupled noisy phase oscillators.
\newblock \emph{Nonlinearity}, 25\penalty0 (5):\penalty0 1247--1273, 2012.
\newblock ISSN 0951-7715.
\newblock \doi{10.1088/0951-7715/25/5/1247}.

\bibitem[Graham and Talay(2013)]{MR3097957}
C.~Graham and D.~Talay.
\newblock \emph{Stochastic simulation and {M}onte {C}arlo methods}, volume~68 of \emph{Stochastic Modelling and Applied Probability}.
\newblock Springer, Heidelberg, 2013.
\newblock \doi{10.1007/978-3-642-39363-1}.
\newblock URL \url{https://doi.org/10.1007/978-3-642-39363-1}.
\newblock Mathematical foundations of stochastic simulation.

\bibitem[Gripenberg et~al.(1990)Gripenberg, Londen, and Staffans]{MR1050319}
G.~Gripenberg, S.-O. Londen, and O.~Staffans.
\newblock \emph{Volterra integral and functional equations}, volume~34 of \emph{Encyclopedia of Mathematics and its Applications}.
\newblock Cambridge University Press, Cambridge, 1990.
\newblock ISBN 0-521-37289-5.
\newblock \doi{10.1017/CBO9780511662805}.
\newblock URL \url{https://doi.org/10.1017/CBO9780511662805}.

\bibitem[Guillin and Monmarch\'{e}(2021)]{MR4333408}
A.~Guillin and P.~Monmarch\'{e}.
\newblock Uniform long-time and propagation of chaos estimates for mean field kinetic particles in non-convex landscapes.
\newblock \emph{J. Stat. Phys.}, 185\penalty0 (2):\penalty0 Paper No. 15, 20, 2021.
\newblock ISSN 0022-4715.
\newblock \doi{10.1007/s10955-021-02839-6}.
\newblock URL \url{https://doi.org/10.1007/s10955-021-02839-6}.

\bibitem[Hauray and Jabin(2007)]{MR2278413}
M.~Hauray and P.-E. Jabin.
\newblock {$N$}-particles approximation of the {V}lasov equations with singular potential.
\newblock \emph{Arch. Ration. Mech. Anal.}, 183\penalty0 (3):\penalty0 489--524, 2007.
\newblock ISSN 0003-9527.
\newblock \doi{10.1007/s00205-006-0021-9}.
\newblock URL \url{https://doi.org/10.1007/s00205-006-0021-9}.

\bibitem[Henry(1981)]{MR610244}
D.~Henry.
\newblock \emph{Geometric theory of semilinear parabolic equations}, volume 840 of \emph{Lecture Notes in Mathematics}.
\newblock Springer-Verlag, Berlin-New York, 1981.
\newblock ISBN 3-540-10557-3.

\bibitem[Herrmann and Tugaut(2010)]{MR2639745}
S.~Herrmann and J.~Tugaut.
\newblock Non-uniqueness of stationary measures for self-stabilizing processes.
\newblock \emph{Stochastic Process. Appl.}, 120\penalty0 (7):\penalty0 1215--1246, 2010.
\newblock ISSN 0304-4149.
\newblock \doi{10.1016/j.spa.2010.03.009}.
\newblock URL \url{https://doi.org/10.1016/j.spa.2010.03.009}.

\bibitem[Journel and Monmarch{\'e}(2022)]{MonmarcheJournel}
L.~Journel and P.~Monmarch{\'e}.
\newblock Convergence of a particle approximation for the quasi-stationary distribution of a diffusion process: uniform estimates in a compact soft case.
\newblock \emph{ESAIM, Probab. Stat.}, 26:\penalty0 1--25, 2022.
\newblock ISSN 1292-8100.
\newblock \doi{10.1051/ps/2021017}.

\bibitem[Leli{\`e}vre and Stoltz(2016)]{greenkubo}
T.~Leli{\`e}vre and G.~Stoltz.
\newblock Partial differential equations and stochastic methods in molecular dynamics.
\newblock \emph{Acta Numerica}, 25:\penalty0 681--880, 2016.
\newblock ISSN 0962-4929.
\newblock \doi{10.1017/S0962492916000039}.

\bibitem[L{\"o}cherbach and Monmarch{\'e}(2022)]{lcherbach2020metastability}
E.~L{\"o}cherbach and P.~Monmarch{\'e}.
\newblock Metastability for systems of interacting neurons.
\newblock \emph{Ann. Inst. Henri Poincar{\'e}, Probab. Stat.}, 58\penalty0 (1):\penalty0 343--378, 2022.
\newblock ISSN 0246-0203.
\newblock \doi{10.1214/21-AIHP1164}.

\bibitem[Lu\c{c}on and Poquet(2017)]{MR3689966}
E.~Lu\c{c}on and C.~Poquet.
\newblock Long time dynamics and disorder-induced traveling waves in the stochastic {K}uramoto model.
\newblock \emph{Ann. Inst. Henri Poincar\'{e} Probab. Stat.}, 53\penalty0 (3):\penalty0 1196--1240, 2017.
\newblock ISSN 0246-0203.
\newblock \doi{10.1214/16-AIHP753}.
\newblock URL \url{https://doi.org/10.1214/16-AIHP753}.

\bibitem[Lu\c{c}on and Poquet(2021)]{MR4254489}
E.~Lu\c{c}on and C.~Poquet.
\newblock Periodicity induced by noise and interaction in the kinetic mean-field {F}itz{H}ugh-{N}agumo model.
\newblock \emph{Ann. Appl. Probab.}, 31\penalty0 (2):\penalty0 561--593, 2021.
\newblock ISSN 1050-5164.
\newblock \doi{10.1214/20-aap1598}.
\newblock URL \url{https://doi.org/10.1214/20-aap1598}.

\bibitem[Ma and Wang(2005)]{MR2310258}
T.~Ma and S.~Wang.
\newblock \emph{Bifurcation theory and applications}, volume~53 of \emph{World Scientific Series on Nonlinear Science. Series A: Monographs and Treatises}.
\newblock World Scientific Publishing Co. Pte. Ltd., Hackensack, NJ, 2005.
\newblock ISBN 981-256-287-7.
\newblock \doi{10.1142/9789812701152}.
\newblock URL \url{https://doi.org/10.1142/9789812701152}.

\bibitem[Malrieu(2003)]{MR1970276}
F.~Malrieu.
\newblock Convergence to equilibrium for granular media equations and their {E}uler schemes.
\newblock \emph{Ann. Appl. Probab.}, 13\penalty0 (2):\penalty0 540--560, 2003.
\newblock ISSN 1050-5164.
\newblock \doi{10.1214/aoap/1050689593}.
\newblock URL \url{https://doi.org/10.1214/aoap/1050689593}.

\bibitem[Reed and Simon(1980)]{SimonBarryBook}
M.~Reed and B.~Simon.
\newblock Methods of modern mathematical physics. {I}: {Functional} analysis. {Rev}. and enl. ed.
\newblock New {York} etc.: {Academic} {Press}, {A} {Subsidiary} of {Harcourt} {Brace} {Jovanovich}, {Publishers}, {XV}, 400 p. \$ 24.00 (1980)., 1980.

\bibitem[Scheutzow(1985)]{MR781411}
M.~Scheutzow.
\newblock Some examples of nonlinear diffusion processes having a time-periodic law.
\newblock \emph{Ann. Probab.}, 13\penalty0 (2):\penalty0 379--384, 1985.
\newblock ISSN 0091-1798.
\newblock URL \url{http://links.jstor.org/sici?sici=0091-1798(198505)13:2<379:SEONDP>2.0.CO;2-L&origin=MSN}.

\bibitem[Simon(1977)]{simonInfiniteDeterminant}
B.~Simon.
\newblock Notes on infinite determinants of hilbert space operators.
\newblock \emph{Advances in Mathematics}, 24\penalty0 (3):\penalty0 244--273, 1977.
\newblock ISSN 0001-8708.
\newblock \doi{https://doi.org/10.1016/0001-8708(77)90057-3}.
\newblock URL \url{https://www.sciencedirect.com/science/article/pii/0001870877900573}.

\bibitem[Sznitman(1991)]{MR1108185}
A.-S. Sznitman.
\newblock Topics in propagation of chaos.
\newblock In \emph{\'{E}cole d'\'{E}t\'{e} de {P}robabilit\'{e}s de {S}aint-{F}lour {XIX}---1989}, volume 1464 of \emph{Lecture Notes in Math.}, pages 165--251. Springer, Berlin, 1991.
\newblock \doi{10.1007/BFb0085169}.
\newblock URL \url{https://doi.org/10.1007/BFb0085169}.

\bibitem[Talay(1990)]{Talay90}
D.~Talay.
\newblock Second-order discretization schemes of stochastic differential systems for the computation of the invariant law.
\newblock \emph{Stochastics Stochastics Rep.}, 29\penalty0 (1):\penalty0 13--36, 1990.
\newblock ISSN 1045-1129.
\newblock \doi{10.1080/17442509008833606}.
\newblock URL \url{hal.inria.fr/inria-00075799/file/RR-0753.pdf}.

\bibitem[Tamura(1984)]{MR0743525}
Y.~Tamura.
\newblock On asymptotic behaviors of the solution of a nonlinear diffusion equation.
\newblock \emph{J. Fac. Sci. Univ. Tokyo Sect. IA Math.}, 31\penalty0 (1):\penalty0 195--221, 1984.
\newblock ISSN 0040-8980.

\bibitem[Tugaut(2013)]{MR3098681}
J.~Tugaut.
\newblock Convergence to the equilibria for self-stabilizing processes in double-well landscape.
\newblock \emph{Ann. Probab.}, 41\penalty0 (3A):\penalty0 1427--1460, 2013.
\newblock ISSN 0091-1798.
\newblock \doi{10.1214/12-AOP749}.
\newblock URL \url{https://doi.org/10.1214/12-AOP749}.

\bibitem[Tugaut(2014)]{zbMATH06330199}
J.~Tugaut.
\newblock Phase transitions of {McKean}-{Vlasov} processes in double-wells landscape.
\newblock \emph{Stochastics}, 86\penalty0 (2):\penalty0 257--284, 2014.
\newblock ISSN 1744-2508.
\newblock \doi{10.1080/17442508.2013.775287}.

\bibitem[Veretennikov(2006)]{MR2208726}
A.~Y. Veretennikov.
\newblock On ergodic measures for {M}c{K}ean-{V}lasov stochastic equations.
\newblock In \emph{Monte {C}arlo and quasi-{M}onte {C}arlo methods 2004}, pages 471--486. Springer, Berlin, 2006.
\newblock \doi{10.1007/3-540-31186-6\_29}.
\newblock URL \url{https://doi.org/10.1007/3-540-31186-6_29}.

\bibitem[Wolansky(1999)]{MR1668556}
G.~Wolansky.
\newblock On nonlinear stability of polytropic galaxies.
\newblock \emph{Ann. Inst. H. Poincar\'{e} Anal. Non Lin\'{e}aire}, 16\penalty0 (1):\penalty0 15--48, 1999.
\newblock ISSN 0294-1449.
\newblock \doi{10.1016/S0294-1449(99)80007-9}.
\newblock URL \url{https://doi.org/10.1016/S0294-1449(99)80007-9}.

\end{thebibliography}
\end{document}